\documentclass[11pt,a4paper,oneside,english]{amsart}

\usepackage{geometry}
\usepackage{xcolor}
\usepackage{lineno}
\usepackage[english]{babel}
\usepackage{amssymb,enumerate,bbm,amsmath}
\usepackage{tikz}
\usepackage{indentfirst}
\usepackage{booktabs}

\numberwithin{equation}{section}
\newtheorem{theorem}{Theorem}[section]
\newtheorem{corollary}[theorem]{Corollary}
\newtheorem{lemma}[theorem]{Lemma}
\newtheorem{proposition}[theorem]{Proposition}
\newtheorem{conjecture}[theorem]{Conjecture}
\newtheorem{remark}[theorem]{Remark}

\newtheorem{question}[theorem]{Question}

\newcommand{\NH}{\mathcal{N}_H}
\newcommand{\V}[1]{\mathcal{V}_{#1}}
\newcommand{\ip}[2]{\langle #1,#2\rangle}
\newcommand{\norm}[1]{\lVert #1\rVert}
\newcommand{\orthsum}{\mathbin{\oplus^{\perp}}}

\newcommand{\RR}{\mathbb{R}}
\newcommand{\Sym}{\operatorname{Sym}}
\newcommand{\tr}{\operatorname{tr}}
\newcommand{\dist}{\operatorname{dist}}
\newcommand{\Cbar}{\bar C}
\newcommand{\Phiv}{\Phi}
\newcommand{\mathd}{\mathrm{d}}

\usepackage{hyperref}
\usepackage{cleveref} 
\hypersetup{colorlinks,linkcolor={blue},citecolor={blue},urlcolor={blue}}
\allowdisplaybreaks
\begin{document}
 
\title[The pinching constant for closed minimal submanifolds]{The pinching constant for closed minimal submanifolds of high codimension in the sphere}

\author{Hongwei Xu}
\address{Center of Mathematical Sciences, Zhejiang University, Hangzhou, 310027, People's Republic of China}
\email{xuhw@zju.edu.cn}

\author{Entao Zhao}
\address{Center of Mathematical Sciences, Zhejiang University, Hangzhou, 310027, People's Republic of China}
\email{zhaoet@zju.edu.cn}

\begin{abstract}
Let $M^n$ be a closed minimal submanifold in the unit sphere $\mathbb{S}^{n+q}$ with $n\geqslant 3$ and $q\geqslant 2$. Let $S$ be the squared length of its second
fundamental form and $S_{\max}=\max_{p\in M}S(p)$. We prove that if $M$ is not totally geodesic, then 
\[
S_{\max}>\frac{2n}{3}+\frac{n-2}{182}.
\]
\end{abstract}

\date{\today}

{\maketitle}

\tableofcontents

\section{Main theorem}

In 1968, Simons proved a famous rigidity theorem for closed minimal submanifolds in the unit sphere.

\begin{theorem} [Simons \cite{J.Simons1968}]
Let $ M^n $ be an  $n$-dimensional closed minimal submanifold in the unit sphere $\mathbb{S}^{n+q}$. Denote by $S$ the squared length of the second fundamental form of $M$. Then the following inequality holds:
\begin{equation*}
\int_{M^n} S \left( S - \frac{n}{2-1/q} \right)  \geqslant 0.
\end{equation*}
In particular, if $0 \leqslant S \leqslant \frac{n}{2-1/q}$, then either $S\equiv0$ and $M^n$ is the totally geodesic sphere in $\mathbb{S}^{n+q}$, or $S\equiv \frac{n}{2-1/q}$.
\end{theorem} 

Later, for closed minimal hypersurfaces in the unit sphere $\mathbb{S}^{n+1}$,  Lawson {\cite{Lawson}}  and Chern--do Carmo--Kobayashi {\cite{Chern}} independently proved that if $S\equiv n$, then $M^n$ is one of the Clifford minimal hypersurfaces $ \mathbb{S}^k\big(\sqrt{k/n}\big) \times \mathbb{S}^{n-k}\big(\sqrt{{(n-k)}/{n}}\big) $ in $ \mathbb{S}^{n+1} $, $ 1 \leqslant k \leqslant n-1 $.  Chern--do Carmo--Kobayashi {\cite{Chern}}  also proved that if $q\geqslant 2$ and $S\equiv {n}/{(2-1/q)}$, then  $n=2$, $q=2$ and $M$ is the Veronese surface in $\mathbb{S}^{4}$. Afterwards, Li--Li \cite{LiLi}  improved the pinching constant to $2n/3$ for codimension $q\geqslant 2$. Chen--Xu \cite{ChenXu} also gave a proof for the rigidity result by using a different argument. In 1990, Xu \cite{Xu-phd,Xu-archmath} proved the generalized Simons-Lawson-Chern-do Carmo-Kobayashi-Li-Li theorem for closed submanifolds with parallel mean curvature in a sphere. Further discussions in this direction have been carried out by
Xu and others; see, for example, \cite{Gu-Xu-Xu-Zhao,Shiohama-Xu,Xu-tams}.

The following question was proposed in \cite{Chern-lecture,Chern}.

\begin{question}[\cite{Chern-lecture,Chern}]\label{Chern-question}
Consider closed minimal submanifolds $M^n$ in the $(n+q)$-dimensional unit sphere $\mathbb{S}^{n+q}$ such that $S$ is a constant. It is plausible that the set of values for $S$ is discrete, at least for $S$ not arbitrarily large. If this is the case, an estimate of the value for $S$ next to $\frac{n}{2-1/q}$ should be of interest. 
\end{question}

The discreteness assertion in Question \ref{Chern-question} has traditionally been referred to as the Chern conjecture.

\begin{conjecture}[Chern Conjecture \cite{Chern-lecture,Chern}]
Let $M^n$ be a closed minimal submanifold with constant scalar curvature in the $(n+q)$-dimensional unit
sphere $\mathbb{S}^{n+q}$. Then for each $n$ and $q$, the set of all possible values for $S$ is discrete.
\end{conjecture}

There has been considerable progress on the Chern conjecture, especially in the codimension-one case.  
In 1982, Yau {\cite{Yau}} listed the codimension-one case of the Chern conjecture as one of the 120 unsolved geometric problems.  In 1983, Peng--Terng {\cite{PengTerng1,PengTerng2}} made the first breakthrough towards the Chern conjecture, proving a second gap theorem for closed minimal hypersurfaces with constant scalar curvature in the unit sphere $\mathbb{S}^{n+1}$. More precisely, if $S > n$, then $S > n + \frac{1}{12n}$. Moreover, for $n = 3$, they proved that $S \geqslant 6$ provided $S > 3$.  In 1993, Chang {\cite{Chang1}} completed the proof of the codimension-one case of the Chern conjecture for $n = 3$. In high dimensions, there have been several estimates on the second gap; see e.g., Yang-Cheng {\cite{YangCheng1,YangCheng2,YangCheng3}} and Suh--Yang {\cite{SuhYang}}. Recently, all closed minimal hypersurfaces with constant scalar curvature and either constant third mean curvature or constant Gauss--Kronecker curvature in $\mathbb{S}^{5}$ were classified \cite{Deng,Deng-Kou,Ge-Liu-Luo-Yan,He-Xu-Zhao}. If the scalar curvature is not required to be constant, then the problem becomes more difficult. Peng--Terng {\cite{PengTerng1,PengTerng2}} proved that, for closed minimal hypersurfaces in $\mathbb{S}^{n+1}$, there exists a positive constant $\delta(n)$ depending only on $n$, such that if $n \leqslant S \leqslant n + \delta(n)$, $n \leqslant 5$, then $S \equiv n$.  Later, Wei--Xu {\cite{WeiXu}} extended the result to $n = 6, 7$ and Zhang {\cite{Q.Zhang}} extended the result to  $n \leqslant 8$.  Finally, Ding--Xin {\cite{DingXin}} extended it to all dimensions by establishing a new estimate; in particular, they showed that if the dimension is $n \geqslant 6$, then the pinching constant $\delta(n) = \frac{n}{23}$. After that, Xu--Xu {\cite{XuXu}} improved it to $\delta(n) = \frac{n}{22}$ and Lei--Xu--Xu {\cite{LeiXuXu}} showed $\delta(n) = \frac{n}{18}$. For more results on codimension-one case of the Chern conjecture and its generalization, please consult \cite{deAlmeida,Chang2,TongzhuLi,Cheng-Wei-Ya,Ge-Tang,Gu-Lei-Xu,LeiXuXu-SciChina,LeiXuXu-pams,Lusala1,SuhYang,XiaoLing,Tang-Yang,Tang-Wei-Yan,Tang-Yan,Xu-Tian,XuXu-mathann,XuXu-jfa,XuXu-SciSinMath}, etc.

The Chern conjecture becomes substantially more difficult in higher codimensions. One of the crucial reasons is that the second fundamental form becomes significantly more intricate in higher codimensions. Recently, that some progress has been made on the Chern conjecture. Ge--Li--Zhang \cite{Ge-Li-Zhang} showed that, for a closed minimal submanifold $M^n$ in the unit sphere $\mathbb{S}^{n+q}$ with flat normal bundle, where $n\geqslant 3$ and $q\geqslant 2$, if $S$ is constant and satisfies $0\leqslant S \leqslant \delta(n,q)$, where $\delta(n,q)$ is a certain positive constant depending only on $n$ and $q$, then $M$ is either totally geodesic, or one of the Clifford minimal hypersurfaces  $ \mathbb{S}^k\big(\sqrt{k/n}\big) \times \mathbb{S}^{n-k}\big(\sqrt{{(n-k)}/{n}}\big) $ in $ \mathbb{S}^{n+1} $, $ 1 \leqslant k \leqslant n-1 $. Lei \cite{Lei} proved that an $n$-dimensional $(n\geqslant 3)$ closed minimal submanifold with constant scalar curvature in $\mathbb{S}^{n+q}$ satisfies  $S =0$ or $S> \frac{2n}{3}+ \frac{1}{32n^3}$. For an $n$-dimensional $(n\geqslant 3)$ non-totally geodesic closed minimal submanifold in $\mathbb{S}^{n+q}$ without assuming that $S$ is constant, Lei \cite{Lei} verified that $S_{\max}>\frac{2n}{3}+\frac{22n-64}{39^2n^4q}$, and Li--Zhao \cite{LiZhao} proved that $S_{\max}\geqslant\frac{2n}{3}+\frac{n-2}{6300(39n+8)}$. These two results concern the quantitative first-gap problem near the Li--Li threshold $\frac{2n}{3}$. 

In this paper, we prove the following theorem.

\begin{theorem}\label{thm:main}
Let $M^n$ be an $n$-dimensional closed minimal submanifold in the $(n+q)$-dimensional unit
sphere $\mathbb{S}^{n+q}$, where $n\geqslant  3$ and $q\geqslant 2$. If $M$ is not totally geodesic,
then
\[
 S_{\max}>\frac{2n}{3}+\frac{n-2}{182}.
\]
\end{theorem}

No constancy assumption on $S$ and no flatness assumption on the normal bundle are imposed in Theorem~\ref{thm:main}. Theorem \ref{thm:main} is an improvement of pinching estimates of Lei \cite{Lei} and Li--Zhao \cite{LiZhao}. 

We note that a recent preprint \cite{FiresterTsiamis} provides counterexamples to the classical Chern conjecture in high codimensions for broad ranges of dimensions and codimensions. This development does not, however, affect the first-gap problem considered here, since all the counterexamples constructed therein satisfy $S>n$. On the other hand, an $n$-dimensional Clifford minimal hypersurface in $\mathbb{S}^{n+1}$, viewed via a totally geodesic inclusion as a minimal submanifold of $\mathbb{S}^{n+q}$ for any $q\geqslant 2$, satisfies $S\equiv n$. It is therefore natural to conjecture that every closed non-totally geodesic minimal submanifold
$M^n\subset\mathbb{S}^{n+q}$ with $n\geqslant 3$ and $q\geqslant 2$ satisfies the sharp estimate
\[
    S_{\max}\geqslant \frac{2n}{3}+\frac{n}{3}=n.
\]

The proof of Theorem \ref{thm:main} is based on the Bochner estimate of Lei \cite{Lei} and on the curvature-gap machinery of Li--Zhao \cite{LiZhao}.  The new ingredients are a combined weighted selection of the defect and Hessian terms, and a sharp algebraic estimate for the cubic part of the Ricci-commutation tensor.
\medskip

\section{Preliminaries}

Throughout the paper, all matrix and tensor norms are Frobenius norms. Matrix inner products are $\langle A,B\rangle=\operatorname{tr}(A^TB)$, and tensor norms sum over all ordered indices. We write $[A,B]=AB-BA$. When wedge notation is used below, $u\wedge v$ denotes the skew-symmetric matrix $uv^T-vu^T$.

Let $h$ be the second fundamental form of the minimal submanifold
$M^n$ immersed in the unit sphere $\mathbb{S}^{n+q}$. 
Choose local orthonormal tangent and normal frames. Write $h=(h^\alpha_{ij})$ and $A_\alpha=(h^\alpha_{ij})$ for the corresponding shape operators. For any tuple $A=(A_1,\ldots,A_q)$ of trace-free symmetric matrices, set $S(A)=\sum_\alpha|A_\alpha|^2$. We omit the argument when the tuple is fixed. For the second fundamental form, $S=|h|^2$.

The Simons identity is 
\begin{equation}\label{eq:simons}
\frac12\Delta S
=|\nabla h|^2+nS-\sum_{\alpha,\beta}\langle A_\alpha,A_\beta\rangle^2
-\sum_{\alpha,\beta}|[A_\alpha,A_\beta]|^2 .
\end{equation}
The Li--Li matrix inequality (see \cite{LiLi}) is
\[
\sum_{\alpha,\beta}\langle A_\alpha,A_\beta\rangle^2
+\sum_{\alpha,\beta}|[A_\alpha,A_\beta]|^2
\leqslant \frac32 S^2 .
\]
Set
\[
D=3S^2-2\sum_{\alpha,\beta}\langle A_\alpha,A_\beta\rangle^2
-2\sum_{\alpha,\beta}|[A_\alpha,A_\beta]|^2 .
\]
Then $D\geqslant 0$ by the Li--Li inequality.

Integrating \eqref{eq:simons} over the closed manifold $M$ and using
$\int_M\Delta S=0$, one has
\begin{equation}\label{eq:intid}
\int_M|\nabla h|^2+\frac12\int_M D
=\frac32\int_M S\Big(S-\frac{2n}{3}\Big).
\end{equation}
Set
\[
P=\int_{\{S\geqslant 2n/3\}}S\Big(S-\frac{2n}{3}\Big),
\quad
N=\int_{\{S<2n/3\}}S\Big(\frac{2n}{3}-S\Big).
\]
Then $P\geqslant  N\geqslant 0$, and from \eqref{eq:intid} we obtain
\begin{equation*}
\int_M|\nabla h|^2\leqslant\frac32 P,
\quad
\int_M D\leqslant 3P.
\end{equation*}

We need the following Bochner-type estimate which is due to Lei \cite{Lei}. Only its integral
consequence is used in the proof. 
\begin{theorem}[Lei \cite{Lei}]
Let $M^n$ be a closed minimal submanifold in the unit sphere $\mathbb{S}^{n+q}$,
where $n\geqslant 3$ and $q\geqslant 2$. Then
\[
\frac12\Delta|\nabla h|^2
\geqslant  |\nabla^2 h|^2+(2n+3-8S)|\nabla h|^2 .
\]
In particular,
\begin{equation}\label{eq:bochner}
\int_M|\nabla^2 h|^2\leqslant\int_M(8S-2n-3)|\nabla h|^2 .
\end{equation}
\end{theorem}

\medskip

\section{A key proposition}
The following proposition is the key to improving the pinching increment above $2n/3$ from a bounded quantity to a quantity growing linearly with $n$.

\begin{proposition}\label{prop:select}
Let $g:=S_{\max}-\frac{2n}{3}>0$, let $t>0$, and let $\Lambda\geqslant  S_{\max}$.
Then there exists $p\in M$ with
\[
\frac{2n}{3}\leqslant S(p)\leqslant\frac{2n}{3}+g
\]
such that
\begin{equation}\label{eq:selectcoupled}
 tnD(p)+|\nabla^2h|^2(p)\leqslant n\Cbar gS(p).
\end{equation}
In particular,
\begin{equation*}
D(p)\leqslant\frac{\Cbar}{t}  g  S(p),
\quad
|\nabla^2 h|^2(p)\leqslant\Cbar  n  g  S(p),
\end{equation*}
where
\[
\Cbar=\max\left\{3t,\frac{3(8\Lambda-2n-3)}{2n}\right\}.
\]
\end{proposition}

\begin{proof}
From \eqref{eq:intid},
\[
 \int_M D=3\int_M S\left(S-\frac{2n}{3}\right)
 -2\int_M|\nabla h|^2.
\]
Combining this identity with \eqref{eq:bochner}, we get
\begin{align*}
  \int_M\bigl(tnD+|\nabla^2h|^2\bigr)
  &\leqslant3tn(P-N)
  +\int_M(8S-2n-3-2tn)|\nabla h|^2\\
  &\leqslant3tnP
 +(8\Lambda-2n-3-2tn)\int_M|\nabla h|^2.
\end{align*}
If the last coefficient is nonpositive, the right-hand side is at most
$3tnP$.  If it is positive, since
$\int_M|\nabla h|^2\leqslant3P/2$, the resulting coefficient of $P$ is
$3(8\Lambda-2n-3)/2$.  Thus,
\[
 \int_M\bigl(tnD+|\nabla^2h|^2\bigr)\leqslant n\Cbar P.
\]
Since $P\leqslant g\int_{\{S\geqslant 2n/3\}}S$ and the integrand on the left is nonnegative,
we get
\[
\int_{\{S\geqslant2n/3\}}\big(tnD+|\nabla^2h|^2\big)
\leqslant\int_M\big(tnD+|\nabla^2h|^2\big)
\leqslant n\Cbar  g\int_{\{S\geqslant 2n/3\}}S .
\]
The set
$\{S\geqslant 2n/3\}$ has positive measure since $S_{\max}-\frac{2n}{3}>0$. So, there is $p\in\{S\geqslant 2n/3\}$ such that
\[
tnD(p)+|\nabla^2h|^2(p)\leqslant n\Cbar g S(p).
\]
This is \eqref{eq:selectcoupled}, and $S(p)\geqslant2n/3$. In particular,
\[
D(p)\leqslant\frac{\Cbar}{t}gS(p),\quad
|\nabla^2h|^2(p)\leqslant\Cbar n g S(p),
\]
while $S(p)\leqslant S_{\max}=\frac{2n}{3}+g$.
\end{proof}

\begin{corollary}\label{cor:selectK}
Assume
\[
 0<g\leqslant\frac{n-2}{K}
\]
for $K>0$, take $\Lambda=S_{\max}$ and
\[
 t=\frac53+\frac4K.
\]
Then the point $p$ in Proposition~\ref{prop:select} may be chosen so that
\begin{equation}\label{eq:selectK}
 tnD(p)+|\nabla^2h|^2(p)\leqslant3tngS(p).
\end{equation}
In particular,
\begin{equation*}
 D(p)\leqslant3gS(p),
 \quad
 |\nabla^2h|^2(p)\leqslant\left(5+\frac{12}{K}\right)ngS(p).
\end{equation*}
\end{corollary}

\begin{proof}
We have
\[
 \begin{split}
 8\Lambda-2n-3-2tn
 &\leqslant\frac{10n}{3}+\frac{8(n-2)}K-3
 -2n\left(\frac53+\frac4K\right)\\
 &=-3-\frac{16}{K}<0.
 \end{split}
\]
Hence $\Cbar=3t=5+12/K$, and 
\eqref{eq:selectcoupled} becomes \eqref{eq:selectK}.
\end{proof}

\medskip

\section{Algebraic stability}

We adopt the same notation as in \cite{LiZhao}. Let $O(n)\times O(q)$ act on $q$-tuples of trace-free symmetric matrices by
\[
(P,Q)\cdot(A_1,\ldots,A_q)
=\left(\sum_\beta Q_{1\beta}PA_\beta P^T,\ldots,
\sum_\beta Q_{q\beta}PA_\beta P^T\right).
\]
For $\rho>0$, let $E_\rho$ be the orbit of
\[
\left(\frac{\sqrt\rho}{2}
\begin{pmatrix}1&0\\0&-1\end{pmatrix}\oplus0,
\frac{\sqrt\rho}{2}
\begin{pmatrix}0&1\\1&0\end{pmatrix}\oplus0,
0,\ldots,0\right).
\]
These compact orbits are the Pauli equality orbits of total squared norm $\rho$.

Let $0<\theta<1/4$ and put
\begin{equation*}
 C(\theta):=
 \frac{c_1}{\sqrt{2m_1}}+\frac{c_2}{\sqrt{2m_2}}
 +\frac1{1-s}+1+\frac{\theta}{(1-s)^2},
\end{equation*}
where
\begin{align*}
 s&=\sqrt\theta,
 &\tau&=\frac{\theta}{1-s},
 &m_1&=\frac{1-\tau}{2},
 &m_2&=\frac{1-s-\tau}{2},\\
 w&=\frac12\left(m_2-\frac{3\theta}{8m_1}\right),
 &\eta&=\frac3{8w},
 &a&=\frac{\eta\theta}{m_1},
 &\gamma&=1-\sqrt{a(2-a)},\\
 Q&=\frac3{4m_1}\left(\frac12+\frac1\gamma\right),
 &u&=\frac{Q\theta}{m_2},
 &c_1&=\frac{\eta}{1+\sqrt{1-a/2}},
 &c_2&=\frac{2Q}{1+\sqrt{1-u}}.
\end{align*}
Assume all denominators are nonzero and the radicands are all non-negative.

The following lemma gives a quantitative distance estimate to the Pauli orbit $E_S$.

\begin{lemma}
\label{lem:distS}
Let $A=(A_1,\dots,A_q)$ be a tuple of trace-free symmetric matrices and set
$S=\sum_\alpha|A_\alpha|^2>0$.  For $0<\theta<1/4$, assume that $w>0$, $0<a<1$, and $0<u<1$.
If $D\leqslant\theta S^2$, then
\[
\dist^2(A,E_S)\leqslant\frac{C(\theta)}{S}D .
\]
\end{lemma}
 
\begin{proof}
Following the proof of \cite[Lemma~4.1]{LiZhao}, we obtain a refinement. Rotate the normal frame so
that the Gram matrix is diagonal and write
\[
 x_1\geqslant x_2\geqslant\cdots\geqslant0,
 \quad x_\alpha=|A_\alpha|^2,
 \quad d=x_1-x_2,
 \quad t=\sum_{\alpha\geqslant3}x_\alpha.
\]
Applying Lu's weighted commutator inequality \cite{Lu} as in Li--Zhao, we obtain
\begin{equation}\label{eq:Ddt}
 D\geqslant d^2+t(S-d).
\end{equation}
Consequently,
\begin{equation}\label{eq:dtbounds}
 d\leqslant sS,
 \quad t\leqslant\frac{D}{S-d}
 \leqslant\frac{D}{S(1-s)}\leqslant\tau S,
 \quad x_1\geqslant m_1S,
 \quad x_2\geqslant m_2S.
\end{equation}

Diagonalize $A_1=\operatorname{diag}(\lambda_1,\ldots,\lambda_n)$ and write
$A_2=(b_{ij})$.  Put
\[
 E=2x_1x_2-|[A_1,A_2]|^2\geqslant0.
\]
By the B\"ottcher--Wenzel inequality, one has the estimate
\[
 \sum_{\alpha\geqslant3}x_\alpha^2
 +4\sum_{3\leqslant\alpha<\beta}x_\alpha x_\beta
 =2t^2-\sum_{\alpha\geqslant3}x_\alpha^2\leqslant2t^2.
\]
By the definition of $D$, one has
\[
 4E\leqslant D+2St-d^2-t^2.
\]
By \eqref{eq:Ddt},
\[
 3D-\bigl(D+2St-d^2-t^2\bigr)
 \geqslant3d^2-2dt+t^2=(t-d)^2+2d^2\geqslant0.
\]
Hence,
\begin{equation}\label{eq:Ethreequarters}
 E\leqslant\frac34D.
\end{equation}

For $i<j$, set
\[
 \Delta_{ij}=2x_1-(\lambda_i-\lambda_j)^2\geqslant0.
\]
We have the decomposition
\[
 E=2x_1\sum_i b_{ii}^2+2\sum_{i<j}\Delta_{ij}b_{ij}^2.
\]
By \eqref{eq:dtbounds} and \eqref{eq:Ethreequarters}, one has
\[
 \sum_i b_{ii}^2\leqslant\frac{E}{2x_1}
 \leqslant\frac{3\theta}{8m_1}S.
\]
Thus, with $W=\sum_{i<j}b_{ij}^2$,
\[
 W=\frac12\left(x_2-\sum_i b_{ii}^2\right)
 \geqslant wS.
\]

A weighted-average argument supplies a pair, relabelled $(1,2)$, for which
\begin{equation}\label{eq:Deltabound}
 \Delta:=\Delta_{12}\leqslant\frac{E}{2W}
 \leqslant\eta\frac DS,
 \quad
 \frac{\Delta}{x_1}\leqslant a.
\end{equation}

We use twice the following elementary radial identity.  If $u$ is a unit
vector, $|v|=r>0$, $c=\langle v,u\rangle\geqslant0$, and
$q=r^2-c^2$, then
\begin{equation}\label{eq:radialidentity}
 |v-ru|^2=\frac{2q}{1+\sqrt{1-q/r^2}}.
\end{equation}
Let $P_1,P_2$ be the two Pauli matrices, with signs chosen according to
$\lambda_1-\lambda_2$ and $b_{12}$, respectively.  Applying
\eqref{eq:radialidentity} to $A_1$, where $q=\Delta/2$, we get
\begin{equation}\label{eq:e1sharp}
 \left|A_1-\sqrt{\frac{x_1}{2}}P_1\right|^2
 \leqslant c_1\frac DS.
\end{equation}

To estimate $A_2$, observe that \eqref{eq:Deltabound} implies
\[
 \Delta_{ij}\geqslant\gamma x_1,\quad \{i,j\}\ne\{1,2\}.
\]
In fact, $|\lambda_1|$ and $|\lambda_2|$ are at least
$\frac12(\sqrt{2-a}-\sqrt a)\sqrt{x_1}$, and the identity
\[
 \Delta_{ij}=(\lambda_i+\lambda_j)^2
 +2\sum_{k\ne i,j}\lambda_k^2
\]
proves the claim.  
It follows that the squared norm $q_2$ of the
component of $A_2$ orthogonal to the $(1,2)$ Pauli direction satisfies
\[
 q_2=x_2-2b_{12}^2
 \leqslant Q\frac DS.
\]
Moreover, $q_2/x_2\leqslant u$.  Applying
\eqref{eq:radialidentity} again, we get
\begin{equation}\label{eq:e2sharp}
 \left|A_2-\sqrt{\frac{x_2}{2}}P_2\right|^2
 \leqslant c_2\frac DS.
\end{equation}

Suppose $B\in E_S$ such that $\dist(A,E_S)=\dist(A,B)=|A-B|$.
For $i=1,2$, set $C_i=\sqrt{x_i/2}P_i$ and $B_i=\sqrt{S/4}P_i$. Then $|A_i|=|C_i|$. Therefore,
\begin{equation}\label{eq:exactrescale}
 |A_i-B_i|^2
 \leqslant \sqrt{\frac{S}{2x_i}}\,|A_i-C_i|^2
 +\left(\sqrt{x_i}-\sqrt{\frac S2}\right)^2.
\end{equation}
By \eqref{eq:dtbounds},
\[
 t\leqslant\frac1{1-s}\frac DS.
\]
Therefore,
\[
 \sum_{i=1}^2\left(\sqrt{x_i}-\sqrt{\frac S2}\right)^2
 \leqslant\frac{d^2+t^2}{S}
 \leqslant\left(1+\frac{\theta}{(1-s)^2}\right)\frac DS.
\]
Combining with \eqref{eq:dtbounds}, \eqref{eq:e1sharp}, \eqref{eq:e2sharp}, and
\eqref{eq:exactrescale}, we conclude that
\[
 \dist^2(A,E_S)\leqslant C(\theta)\frac DS.
\]
This proves the lemma.
\end{proof}

We next rescale the comparison tuple to the critical orbit $E_{2n/3}$.

\begin{lemma}
    \label{lem:dist}
Let $A=(A_1,\dots,A_q)$ be a tuple of trace-free symmetric matrices and set
$S=\sum_\alpha|A_\alpha|^2>0$.  Let $\theta$ satisfy the admissibility
conditions of Lemma~\ref{lem:distS}.  If $D\leqslant\theta S^2$ and
$S\geqslant2n/3$, then
\begin{equation}\label{eq:dist}
\dist^2(A,E_{2n/3})
\leqslant\frac{C(\theta)}{S}D
+\Big(\sqrt S-\sqrt{\frac{2n}{3}}\Big)^2 .
\end{equation}
\end{lemma}

\begin{proof}
By Lemma~\ref{lem:distS}, there is $B\in E_S$ with
$|A-B|=\dist(A,E_S)$ and
$|A-B|^2\leqslant\frac{C(\theta)}{S}D$. Put
$\lambda=\sqrt{\frac{2n}{3}}/\sqrt S\leqslant1$ and $C=\lambda B$. Then
$C\in E_{2n/3}$ and
\[
|B-C|^2=(1-\lambda)^2|B|^2=\Big(\sqrt S-\sqrt{\frac{2n}{3}}\Big)^2 .
\]
Since $|B|=|A|=\sqrt S$, by the Cauchy--Schwarz inequalty,
\[
\langle A-B,B\rangle=\langle A,B\rangle-S\leqslant0 .
\]
Therefore,
\[
|A-C|^2
=|A-B|^2+|B-C|^2+2(1-\lambda)\langle A-B,B\rangle
\leqslant|A-B|^2+|B-C|^2 .
\]
Consequently,
\[
\dist^2(A,E_{2n/3})\leqslant|A-C|^2
\leqslant\frac{C(\theta)}{S}D
+\Big(\sqrt S-\sqrt{\frac{2n}{3}}\Big)^2 .
\qedhere
\]
\end{proof}

\begin{remark}\label{remark}
$(1)$ The constant $\theta=1/64$ is used in \cite{LiZhao}, and the coefficient of $D/S$ in that estimate is $20$ rather than $C(1/64)$ where
\[
 C(1/64)=6.4756606507\ldots<\frac{13}{2}.
\]
$(2)$ Set  $\theta_*:=\frac9{364}$. On the whole interval $0<\theta\leqslant\theta_*$ one has
\[
 w>0.19380,
 \quad 0<a<0.09858,
 \quad 0<u<0.21595.
\]
So, all admissibility conditions in Lemma~\ref{lem:distS} hold. Moreover, at the endpoint one has
\begin{equation*}
 C(\theta_*)=7.3256369876\ldots<7.326.
\end{equation*}
\end{remark}

\begin{remark}
The hypothesis $S\geqslant 2n/3$ is used exactly to make the cross term
$2(1-\lambda)\langle A-B,B\rangle$ nonpositive. In the application below the
selected point satisfies $S(p)\geqslant 2n/3$, so \eqref{eq:dist} applies.
\end{remark}

\section{The Ricci-commutation obstruction}
As in \cite{LiZhao}, for an algebraic tuple $a=(a^\alpha_{ij})\in(\Sym^2_0(\RR^n))^q$, let $R(a)$
and $R^\perp(a)$ denote the algebraic Gauss and Ricci tensors obtained by
substituting $a$ for $h$:
\[
R_{ijkl}(a)
=\delta_{ik}\delta_{jl}-\delta_{il}\delta_{jk}
+\sum_\alpha\big(a^\alpha_{ik}a^\alpha_{jl}-a^\alpha_{il}a^\alpha_{jk}\big),
\]
\[
R^\perp_{\alpha\beta kl}(a)=[A_\alpha,A_\beta]_{kl},
\quad A_\alpha=(a^\alpha_{ij}).
\]
Define
\[
T_1(u,K)_{\alpha ijkl}=\sum_r u^\alpha_{rj}K_{rikl},
\]
\[
T_2(u,K)_{\alpha ijkl}=\sum_r u^\alpha_{ir}K_{rjkl},
\]
\[
N(u,L)_{\alpha ijkl}=\sum_\beta u^\beta_{ij}L_{\beta\alpha kl}.
\]
The maps $T_1,T_2$ and $N$ are bilinear in their two arguments.

Define the Ricci-commutation tensor by
\[
\Phiv(a)=T_1(a,R(a))+T_2(a,R(a))+N(a,R^\perp(a)).
\]
For an actual immersion, by the Ricci identity for the second covariant differentiation of the second fundamental form $h$, there holds
\begin{equation}\label{eq:ricci}
|\Phiv(h)|^2\leqslant4|\nabla^2 h|^2 .
\end{equation}

Write $R(a)=R^0+G(a)$ with

\[
R^0_{ijkl}=\delta_{ik}\delta_{jl}-\delta_{il}\delta_{jk},\quad
G(a)_{ijkl}=\sum_\gamma\big(a^\gamma_{ik}a^\gamma_{jl}-a^\gamma_{il}a^\gamma_{jk}\big),
\]
and put
\[
\Phiv^0(a)=T_1(a,R^0)+T_2(a,R^0).
\]
Since $R^0$ is fixed, $\Phi^0$ is linear.

We need the following exact computation.

\begin{lemma}\label{lem:flat-ricci-norm}
For every trace-free symmetric tuple $h$, one has the exact identity
\[
|\Phiv^0(h)|^2=4nS .
\]
Consequently, for any trace-free symmetric tuples $h,\tilde h$,
\[
|\Phiv^0(h)-\Phiv^0(\tilde h)|=2\sqrt n\,|h-\tilde h| .
\]
\end{lemma}

\begin{proof}
Substituting $R^0_{ijkl}=\delta_{ik}\delta_{jl}-\delta_{il}\delta_{jk}$ into the
definition of $T_1,T_2$, we get
\[
(\Phiv^0(h))^\alpha_{ijkl}
=h^\alpha_{kj}\delta_{il}-h^\alpha_{lj}\delta_{ik}
+h^\alpha_{ik}\delta_{jl}-h^\alpha_{il}\delta_{jk}.
\]

Fix a normal index $\alpha$ and write
\[
(\Phiv^0(h))^\alpha_{ijkl}
=\Psi^1_{ijkl}+\Psi^2_{ijkl}
 +\Psi^3_{ijkl}+\Psi^4_{ijkl},
\]
where
\[
\begin{aligned}
\Psi^1_{ijkl}&= h^\alpha_{kj}\delta_{il},&
\Psi^2_{ijkl}&=-h^\alpha_{lj}\delta_{ik},\\
\Psi^3_{ijkl}&= h^\alpha_{ik}\delta_{jl},&
\Psi^4_{ijkl}&=-h^\alpha_{il}\delta_{jk}.
\end{aligned}
\]
Squaring and summing over
$i,j,k,l$, each of the four self-products equals
\[
\sum_{i,j,k,l}(\Psi^1_{ijkl})^2=n\sum_{j,k}\big(h^{\alpha}_{jk}\big)^2=n\,S_\alpha,
\]
and likewise for $\Psi^2,\Psi^3,\Psi^4$, where $S_\alpha=\sum_{j,k}(h^\alpha_{jk})^2$. Hence the
four squares contribute $4nS_\alpha$ in total. The six cross terms are
\[
\sum_{ijkl}\Psi^1_{ijkl}\Psi^2_{ijkl}=-S_\alpha,\quad
\sum_{ijkl}\Psi^1_{ijkl}\Psi^3_{ijkl}=S_\alpha,
\]
\[
\sum_{ijkl}\Psi^2_{ijkl}\Psi^4_{ijkl}=S_\alpha,\quad
\sum_{ijkl}\Psi^3_{ijkl}\Psi^4_{ijkl}=-S_\alpha.
\]
So the pairs $\Psi^1\Psi^2$ and $\Psi^1\Psi^3$ cancel, as do $\Psi^2\Psi^4$ and $\Psi^3\Psi^4$. Finally,
\[
\sum_{ijkl}\Psi^1_{ijkl}\Psi^4_{ijkl}=-\Big(\sum_i h^\alpha_{ii}\Big)^2=0,\quad
\sum_{ijkl}\Psi^2_{ijkl}\Psi^3_{ijkl}=-\Big(\sum_i h^\alpha_{ii}\Big)^2=0,
\]
by the trace-freeness $\sum_i h^\alpha_{ii}=0$. Summing over $\alpha$ gives
\[
|\Phiv^0(h)|^2=\sum_{\alpha}4n\,S_\alpha=4nS .
\]
The Lipschitz estimate follows from the linearity of $\Phiv^0$:
$$
|\Phiv^0(h)-\Phiv^0(\tilde h)|=|\Phiv^0(h-\tilde h)|=2\sqrt n\,|h-\tilde h|,$$
since $h-\tilde h$ is again trace-free and symmetric.
\end{proof}

Put
\[
 \mathcal F(a):=\Phiv(a)-\Phiv^0(a)
 =T_1(a,G(a))+T_2(a,G(a))+N(a,R^\perp(a))
\]
and
\[
 \mathcal Q(a):=
 \sum_{\alpha,\beta}\langle A_\alpha,A_\beta\rangle^2
 +\sum_{\alpha,\beta}|[A_\alpha,A_\beta]|^2.
\]

\begin{lemma}
    \label{lem:cubic}
For every trace-free tuple $a$ of real symmetric matrices, one has
\begin{equation}\label{eq:slot}
 2|T_1(a,G(a))|^2+|N(a,R^\perp(a))|^2
 \leqslant S\mathcal Q(a).
\end{equation}
Consequently,
\begin{equation}\label{eq:cubicdiag}
 |\mathcal F(a)|^2\leqslant3S\mathcal Q(a)
 \leqslant\frac92S^3.
\end{equation}
Moreover, the constant in \eqref{eq:slot} is sharp.
\end{lemma}

\begin{proof}
Write
\[
 \Gamma_{\alpha\beta}=\langle A_\alpha,A_\beta\rangle,
 \quad I=|\Gamma|^2,
 \quad J=\sum_{\alpha,\beta}|[A_\alpha,A_\beta]|^2.
\]
For a unit vector $x\in\RR^n$, set
\[
 K(x)=\sum_{i,k,l}G(x,e_i,e_k,e_l)^2,
 \quad u_\alpha=A_\alpha x,
 \quad U_{\alpha\beta}=\langle u_\alpha,u_\beta\rangle.
\]
Viewing $G(x,e_i,\cdot,\cdot)$ as a 2-tensor, one has
\[
G(x,e_i,\cdot,\cdot)=\sum_\alpha (A_\alpha x)\wedge (A_\alpha e_i).
\]
Using
\[
 \langle u\wedge v,p\wedge q\rangle
 =2\big(\langle u,p\rangle\langle v,q\rangle
 -\langle u,q\rangle\langle v,p\rangle\big),
\]
we obtain by a direct expansion that
\begin{equation}\label{eq:Kexpand}
 K(x)=2\langle\Gamma,U\rangle-2Z+J_x,
\end{equation}
where
\[
 Z=\sum_{\alpha,\beta}|A_\alpha A_\beta x|^2,
 \quad
 J_x=\sum_{\alpha,\beta}|[A_\alpha,A_\beta]x|^2.
\]

Let $P_x$ denote orthogonal projection onto $x^\perp$ and put
$v_{\alpha\beta}=P_x(A_\alpha A_\beta x)$. By the symmetry of the shape
operators,
\[
 Z-|U|^2=\sum_{\alpha,\beta}|v_{\alpha\beta}|^2,
 \quad
 [A_\alpha,A_\beta]x=v_{\alpha\beta}-v_{\beta\alpha}.
\]
It follows that
\begin{equation}\label{eq:Jxprojection}
 J_x\leqslant4\big(Z-|U|^2\big).
\end{equation}
Combining \eqref{eq:Kexpand} and \eqref{eq:Jxprojection} we obtain
\[
 \begin{split}
 K(x)
 &\leqslant2\langle\Gamma,U\rangle-2|U|^2+\frac12J_x\\
 &=\frac12I-2\left|U-\frac12\Gamma\right|^2+\frac12J_x.
 \end{split}
\]
Every commutator $[A_\alpha,A_\beta]$ is skew-symmetric.  Hence, for
$|x|=1$,
\[
 |[A_\alpha,A_\beta]x|^2
 \leqslant\frac12|[A_\alpha,A_\beta]|^2.
\]
Therefore,
\begin{equation}\label{eq:directionalG}
 K(x)\leqslant\frac12I+\frac14J.
\end{equation}

Define the positive semidefinite matrices
\[
 H=\sum_\alpha A_\alpha^2,
 \quad
 \mathcal K_{rs}=\sum_{i,k,l}G_{rikl}G_{sikl}.
\]
Since $x^T\mathcal Kx=K(x)$, we obtain by \eqref{eq:directionalG} that
$\mathcal K\leqslant(I/2+J/4)\operatorname{Id}$.  Moreover,
\[
 |T_1(a,G(a))|^2=\operatorname{tr}(H\mathcal K),
 \quad \operatorname{tr}H=S.
\]
Thus,
\begin{equation}\label{eq:T1sharp}
 2|T_1(a,G(a))|^2\leqslant S\left(I+\frac12J\right).
\end{equation}

Rotate the normal frame so that
$\Gamma_{\alpha\beta}=s_\alpha\delta_{\alpha\beta}$, where
$s_\alpha=|A_\alpha|^2$ and $\sum_\alpha s_\alpha=S$.  Then
\[
 \begin{split}
 |N(a,R^\perp(a))|^2
 &=\sum_{\alpha,\beta}s_\beta|[A_\beta,A_\alpha]|^2\\
 &=\frac12\sum_{\alpha,\beta}(s_\alpha+s_\beta)
 |[A_\alpha,A_\beta]|^2
 \leqslant\frac S2J.
 \end{split}
\]
Together with \eqref{eq:T1sharp}, this proves \eqref{eq:slot}.
Since $|T_2|=|T_1|$, the elementary three-term inequality and the Li--Li
inequality give \eqref{eq:cubicdiag}.

For sharpness, take
\[
 A_1=c\begin{pmatrix}1&0\\0&-1\end{pmatrix}\oplus0,
 \quad
 A_2=c\begin{pmatrix}0&1\\1&0\end{pmatrix}\oplus0,
 \quad 
 A_3=\dots=A_q=0.
\]
By a direct computation,
$S=4c^2$, $I=8c^4$, $J=16c^4$, and
$|T_1|^2=|N|^2=32c^6$. So the equality holds in \eqref{eq:slot}.
\end{proof}

As $\mathcal F$ is a cubic homogeneous mapping, it has a unique symmetric
trilinear polarization $\mathcal T$ satisfying
$\mathcal F(a)=\mathcal T(a,a,a)$. 
Explicitly,
\[
 \mathcal T (u,v,w)=\frac{1}{6}\Big[
 \mathcal{F} (u+v+w) -\mathcal{F}(u+v)-\mathcal{F}(u+w)-\mathcal{F}(v+w)+\mathcal{F}(u)+\mathcal{F}(v)+\mathcal{F}(w)
 \Big].   
\]

\begin{corollary}\label{cor:trilinear}
For all trace-free symmetric tuples $u,v,w$,
\begin{equation}\label{eq:trilinear}
 |\mathcal T(u,v,w)|\leqslant\frac3{\sqrt2}|u||v||w|.
\end{equation}
\end{corollary}

\begin{proof}
By Lemma~\ref{lem:cubic},
$|\mathcal T(a,a,a)|\leqslant3|a|^3/\sqrt2$.  For every unit vector $\eta$
in the target tensor space, the scalar form
$\langle\mathcal T(\cdot,\cdot,\cdot),\eta\rangle$ is symmetric.  By Banach's theorem for symmetric multilinear forms on a real Hilbert space
(see e.g., \cite{CarandoRodriguez,Pappas-etc}),  its norm is identified with its diagonal norm.  Taking the supremum over $\eta$,  \eqref{eq:trilinear} is proved.
\end{proof}


Put
\[
 \mathcal V:=\bigl(\Sym^2_0(\RR^n)\bigr)^q.
\]
With the Euclidean inner products on the tensor spaces, the differential of
$\Phiv$ at $H\in  \mathcal V$ is the linear map
\[
 \begin{split}
 D\Phiv_H:\mathcal V &\longrightarrow  \RR^q\otimes(\RR^n)^{\otimes4}\\
 \quad d &\longmapsto (D\Phiv_H)[d]
\end{split}
\]
defined by
\[
 (D\Phiv_H)[d]  :=\left.\frac{\mathd}{\mathd s}\right|_{s=0}\Phiv(H+sd)
 =\lim_{s\to0}\frac{\Phiv(H+sd)-\Phiv(H)}{s}.
\]

\begin{lemma}
    \label{lem:spectrum}
Let $h_0\in E_{2n/3}$ and put $L=D\Phiv_{h_0}$.  If
$d\perp T_{h_0}E_{2n/3}$, write
\[
 d=xe_0+z,
 \quad e_0=\frac{h_0}{\sqrt{2n/3}},
 \quad z\perp e_0,\quad z\perp T_{h_0}E_{2n/3}.
\]
Then
\begin{equation}\label{eq:spectral}
 |Ld|^2\leqslant n^2\big(\lambda_nx^2+\kappa_n^2|z|^2\big),
\end{equation}
where
\[
 \lambda_n=18-\frac{20}{n},
 \quad
 \kappa_3^2=\frac43,
 \quad
 \kappa_n^2=\frac{2(n+2)}{3n}\quad(n\geqslant4).
\]
\end{lemma}

\begin{proof} In the present setting,
$\Phiv=\Phiv^0+\mathcal F$ and
$\mathcal F(a)=\mathcal T(a,a,a)$. 
As\[
 \begin{aligned}
 \mathcal F(H+sd)
 ={}& \mathcal T(H+sd,H+sd,H+sd)\\ 
 ={}&\mathcal F(H) +3s\mathcal T(H,H,d)+3s^2\mathcal T(H,d,d)+s^3\mathcal T(d,d,d),
 \end{aligned}
\]
one has
\[
 \Phiv(H+sd)
=\Phiv(H) +s \Big(\Phiv^0(d) +3\mathcal T(H,H,d) \Big)+3s^2\mathcal T(H,d,d)+s^3\mathcal T(d,d,d).
 \]
By the definition of the differential of $\Phiv$ at $H$, this gives
\[
 (D\Phi_H)[d]=\Phiv^0(d)+3\mathcal T(H,H,d).
\]

The curvature terms $G(a)$ and $R^\perp(a)$ are homogeneous of degree two in $a$. 
As in the expansion of $\mathcal F (H+sd)$, one has
\begin{equation}\label{eq:DG_H}
    G(H+sd)=G(H)+s(DG_H)[d]+s^2G(d),
\end{equation}
\begin{equation}\label{eq:DR_H}
R^\perp(H+sd) =R^\perp(H)+s(DR^\perp_H)[d]+s^2R^\perp(d).
\end{equation}

Since $R^0$ is a fixed tensor, $DR_H=DG_H$.

By the definition of $\Phi$, one has
\begin{equation}\label{def-Phi}
\begin{split}
 \Phiv(H+sd)
 =&\Phiv^0(H+sd)+T_1(H+sd,G(H+sd))\\
 &+T_2(H+sd,G(H+sd))+N(H+sd,R^\perp(H+sd)).
\end{split}
\end{equation}

Since $T_1(\cdot , \cdot)$ is bilinear, one has
\[
\frac{\mathd}{\mathd s} T_1(H+sd,G(H+sd))= T_1( d,G(H+sd))+T_1\left(H+sd, \frac{\mathd}{\mathd s}G(H+sd)\right).
\]
By \eqref{eq:DG_H}, one has
\[
\left.\frac{\mathd}{\mathd s}\right|_{s=0} G(H+sd)= (DG_H)[d].
\]
Therefore, 
\[
\left.\frac{\mathd}{\mathd s}\right|_{s=0}T_1(H+sd,G(H+sd))=T_1( d,G(H))+T_1(H, (DG_H)[d]).
\]
Similarly, 
\[
\left.\frac{\mathd}{\mathd s}\right|_{s=0}T_2(H+sd,G(H+sd))= T_2( d,G(H))+T_2(H, (DG_H)[d]),
\]
\[
\left.\frac{\mathd}{\mathd s}\right|_{s=0}N(H+sd,R^\perp(H+sd))= N( d,R^\perp(H))+N(H, (DR^\perp_H)[d]).
\]
So, by differentiating both sides of \eqref{def-Phi}, we get
\begin{equation}\label{eq:normal-linearization}
 \begin{aligned}
  (D\Phi_H)[d]={}&\Phiv^0(d)
 +T_1(d,G(H))+T_2(d,G(H))\\
 &+T_1\bigl(H,(DG_H)[d]\bigr)
 +T_2\bigl(H,(DG_H)[d]\bigr)\\
 &+N(d,R^\perp(H))
 +N\bigl(H,(DR^\perp_H)[d]\bigr).
 \end{aligned}
\end{equation}

Take $H=h_0\in E_{2n/3}$. By $O(n)\times O(q)$-equivariance, choose orthonormal tangent and normal frames in which $h_0$ is the standard Pauli tuple displayed below:
\[
 c:=\sqrt{\frac n6},
 \quad
 P_1:=\begin{pmatrix}1&0\\0&-1\end{pmatrix},
 \quad
 P_2:=\begin{pmatrix}0&1\\1&0\end{pmatrix},
\]
\[
 H_1=cP_1\oplus0,
 \quad H_2=cP_2\oplus0,
 \quad H_\mu=0\quad(\mu\geqslant3).
\]
For $d=(D_\alpha)\in \mathcal{V}$, relative to $\RR^n=\RR^2\oplus\RR^{n-2}$, write
\[
 D_\alpha=
 \begin{pmatrix}
 B_\alpha&C_\alpha\\
 C_\alpha^T&E_\alpha
 \end{pmatrix},
\]
where $B_\alpha\in \Sym^2(\RR^2)$, $C_\alpha\in \RR^{2\times(n-2)}$, $E_\alpha\in \Sym^2(\RR^{n-2})$. Then
 $d\perp T_HE_{2n/3}$
 if and only if
\[
 \begin{cases}
 \langle B_1,P_2\rangle-\langle B_2,P_1\rangle=0,\\
 (D_1)_{1a}+(D_2)_{2a}=0,&a\geqslant3,\\
 (D_2)_{1a}-(D_1)_{2a}=0,&a\geqslant3,\\
 \langle B_\mu,P_1\rangle=\langle B_\mu,P_2\rangle=0,
 &\mu\geqslant3.
 \end{cases}
\]
See Appendix \ref{app:pauli-normal} for the proof.

Let
\[
 \mathcal N_H:=(T_HE_{2n/3})^\perp.
\]
Appendix~\ref{app:normal-equations} gives the orthogonal decomposition
\[
 \mathcal N_H
 =\RR e_0\oplus^\perp\mathcal V_1\oplus^\perp\cdots
 \oplus^\perp\mathcal V_6.
\]
The computations in Appendix~\ref{app:direct-normal-spectrum} show that
\[
 L(\RR e_0),L(\mathcal V_1),\ldots,L(\mathcal V_6)
 \quad\text{are pairwise orthogonal},
\]
and that
\[
 |Le_0|^2=n^2\lambda_n,
 \qquad
 |Lz_j|^2=n^2\mu_j|z_j|^2
 \quad(z_j\in\mathcal V_j),
\]
where
\[
\begin{split}
\mu_1=&\dfrac{2(n+2)}{3n},    \quad  \quad  \ \textrm{dim} \mathcal V_1=n(n-3) ,\\
\mu_2=&\dfrac4n,    \quad  \quad \quad \quad \  \quad \textrm{dim} \mathcal V_2= 2(n-2),\\
\mu_3=&\dfrac{2(n+10)}{9n},    \quad \quad  \textrm{dim} \mathcal V_3= 2,\\
\mu_4=&\dfrac{2(n+6)}{9n},  \quad   \quad \  \textrm{dim} \mathcal V_4= 2,\\
\mu_5=&\dfrac4n,    \quad \quad \quad \quad \quad \  \textrm{dim} \mathcal V_5= \dfrac{(q-2)(n-1)(n-2)}2,\\
\mu_6=&\dfrac{n+4}{3n},  \quad \quad \quad  \ \   \textrm{dim} \mathcal V_6=2(q-2)(n-2) .
\end{split}
\]
Hence,
\[
 \begin{aligned}
 1+\sum_{j=1}^6\dim\mathcal V_j
 ={}&1+n(n-3)+2(n-2)+4\\
 &+(q-2)\frac{(n-1)(n-2)}2+2(q-2)(n-2)\\
 ={}&q\left(\frac{n(n+1)}2-1\right)-(2n+2q-7)
 =\dim\mathcal N_H.
 \end{aligned}
\]
Moreover,
\[
\kappa_n^2:= \max_{1\leqslant j\leqslant6}\mu_j
 =
 \begin{cases}
 \dfrac43,&n=3,\\[1mm]
 \dfrac{2(n+2)}{3n},&n\geqslant4.
 \end{cases}
 \]

Finally, write
\[
 z=\sum_{j=1}^6z_j,
 \qquad z_j\in\mathcal V_j,
\]
so that $d=xe_0+\sum_{j=1}^6z_j$.  Since the decomposition of
$\mathcal N_H$ is orthogonal,
\[
 |z|^2=\sum_{j=1}^6|z_j|^2.
\]
By the linearity of $L$ and the pairwise orthogonality of the images,
\[
 \begin{aligned}
 |Ld|^2
 &=\left|xLe_0+\sum_{j=1}^6Lz_j\right|^2\\
 &=x^2|Le_0|^2+\sum_{j=1}^6|Lz_j|^2\\
 &=n^2\lambda_nx^2+n^2\sum_{j=1}^6\mu_j|z_j|^2\\
 &\leqslant n^2\lambda_nx^2
   +n^2\kappa_n^2\sum_{j=1}^6|z_j|^2\\
 &=n^2\bigl(\lambda_nx^2+\kappa_n^2|z|^2\bigr).
 \end{aligned}
\]
This proves \eqref{eq:spectral}.
\end{proof}

\begin{lemma}
    \label{lem:ricci}
Let $h_0\in E_{2n/3}$ minimize the distance from a trace-free symmetric
tuple $h$ to $E_{2n/3}$, and put $d=h-h_0=xe_0+z$ as in
Lemma~\ref{lem:spectrum}.  If $r=|d|$, then
\begin{equation}\label{eq:improvedRicci}
  |\Phiv(h)-\Phiv(h_0)|
 \leqslant n\sqrt{\lambda_nx^2+\kappa_n^2|z|^2}+3\sqrt{3n}\,r^2+\frac3{\sqrt2}r^3.
\end{equation}
Moreover,
$|\Phiv(h_0)|^2=\frac43n^2(n-2)$.
\end{lemma}

\begin{proof}
The minimizing property implies  $d\perp T_{h_0}E_{2n/3}$.  Since
$\Phiv=\Phiv^0+\mathcal T(\cdot,\cdot,\cdot)$ and $\Phiv^0$ is linear,
the exact Taylor expansion is
\[
 \Phiv(h_0+d)-\Phiv(h_0)
 =Ld+3\mathcal T(h_0,d,d)+\mathcal T(d,d,d).
\]
Now apply Lemma~\ref{lem:spectrum}, Corollary~\ref{cor:trilinear}, and the fact
$|h_0|=\sqrt{2n/3}$.  The quadratic coefficient is
\[
 3\cdot\frac3{\sqrt2}\sqrt{\frac{2n}{3}}=3\sqrt{3n},
\]
which proves \eqref{eq:improvedRicci}.  The remaining value at $h_0$ is
the exact Pauli computation of \cite[Lemma~4.3]{LiZhao}.
\end{proof}

\medskip

\section{Proof of the main theorem}

\begin{proof}[Proof of Theorem~\ref{thm:main}]
Since $M$ is not totally geodesic, by the Li--Li theorem one has
$S_{\max}>2n/3$.  Put
\[
 g=S_{\max}-\frac{2n}{3}>0.
\]
Suppose, for contradiction, that
\[
 g\leqslant\delta:=\frac{n-2}{K},
 \quad K=182.
\]
Let $p$ be the point supplied by Corollary~\ref{cor:selectK}.
Write
\[
 S=S(p),\quad \mathcal H=|\nabla^2h|^2(p),\quad
 X=X_n:=\frac{n-2}{n}.
\]

\medskip
\noindent\textbf{Step 1.} 
Define
\[
 y:=\frac{D(p)}{gS}.
\]
By \eqref{eq:selectK}, one has
\begin{equation}\label{eq:coupledy}
 0\leqslant y\leqslant3,
 \quad \mathcal H\leqslant tngS(3-y),
 \quad t=\frac53+\frac4K.
\end{equation}
If $y>0$, put
\[
 \theta=\theta_{K,X}(y):=\frac{3Xy}{2K}.
\]
Since $S\geqslant2n/3$,
\[
 \frac{D(p)}{S^2}=\frac{yg}{S}
 \leqslant\frac{3Xy}{2K}=\theta
 \leqslant\frac9{2K}=\theta_*.
\]
Thus Lemma~\ref{lem:dist} applies.  The function $C(\theta)$ has the
continuous extension $C(0)=5$; when $y=0$, the same conclusion follows by
applying the lemma with any positive admissible radius and then using $D=0$.

Since
\[
 \left(\sqrt S-\sqrt{\frac{2n}{3}}\right)^2
 \leqslant\frac{3}{8n}\left(S-\frac{2n}{3}\right)^2
 \leqslant\frac{3g^2}{8n},
\]
we obtain
\begin{equation}\label{eq:Bnnew}
 r^2:=\dist^2\big(A(p),E_{2n/3}\big)
 \leqslant B_{K,X}(y)(n-2),
\end{equation}
where
\begin{equation*}
 B_{K,X}(y):=
 \frac{yC(3Xy/(2K))}{K}+\frac{3X}{8K^2}.
\end{equation*}
Choose a nearest $h_0\in E_{2n/3}$ and put
\[
 d=h(p)-h_0=xe_0+z,
 \quad r=|d|.
\]

\medskip
\noindent\textbf{Step 2.}
Put $\rho=2n/3$.  Since
$\rho\leqslant|h(p)|^2\leqslant\rho+\delta$, we have
\[
 0\leqslant2\sqrt\rho\,x+r^2\leqslant\delta.
\]
Set
\begin{equation*}
 A_{K,X}(y):=\max\left\{B_{K,X}(y),\frac1K\right\}.
\end{equation*}
Since both $S(p)-\rho$ and $r^2$ are nonnegative, the identity
$2\sqrt\rho\,x+r^2=S(p)-\rho$, together with \eqref{eq:Bnnew}, implies
\begin{equation}\label{eq:xbound}
 |x|\leqslant\frac{A_{K,X}(y)(n-2)}{2\sqrt\rho},
 \quad |z|^2\leqslant B_{K,X}(y)(n-2).
\end{equation}
Define
\begin{equation*}
 \begin{split}
 \varepsilon_{K,n}(y):={}&
 \sqrt{\kappa_n^2B_{K,X}(y)+
 \frac{3\lambda_nX}{8}A_{K,X}(y)^2}\\
 &+3\sqrt3\,B_{K,X}(y)\sqrt X
 +\frac3{\sqrt2}B_{K,X}(y)^{3/2}X.
 \end{split}
\end{equation*}
By Lemma~\ref{lem:ricci}, \eqref{eq:Bnnew}, and \eqref{eq:xbound}, one has
\[
 |\Phiv(h(p))-\Phiv(h_0)|
 \leqslant n\sqrt{n-2}\,\varepsilon_{K,n}(y).
\]
Consequently,
\begin{equation*}
 |\Phiv(h(p))|
 \geqslant n\sqrt{n-2}
 \left(\sqrt{\frac43}-\varepsilon_{K,n}(y)\right).
\end{equation*}

On the other hand, by \eqref{eq:ricci} and \eqref{eq:coupledy}, we have
\begin{equation*}
 |\Phiv(h(p))|^2
 \leqslant U_{K,X}(y)n^2(n-2)
\end{equation*}
where
\begin{equation*}
 U_{K,X}(y):=
 \frac{4t(3-y)}K\left(\frac23+\frac XK\right).
\end{equation*}

\medskip
\noindent\textbf{Step 3.} To complete the proof of Theorem~\ref{thm:main}, we first prove two lemmas.

\begin{lemma}\label{lem:stability-slope}
The function \(C\) is nondecreasing on \([0,9/364]\). Moreover, if
\[
 c(y)=C(3y/364),
\]
then
\begin{equation}\label{eq:stability-derivative}
 c'(y)>\frac35
 \quad\text{for all }y\in[5/2,3].
\end{equation}
In particular,
\begin{equation}\label{eq:stability-affine}
 C(3y/364)<7.326-\frac35(3-y)
 \quad\text{for all }y\in[5/2,3].
\end{equation}
\end{lemma}

\begin{proof}
On the admissible interval, \(s\) and \(\tau\) are increasing, whereas
\(m_1,m_2,w\) are decreasing. Hence \(\eta\) and \(a\) are increasing.
Since \(0\leqslant a<1\), the function
\(\gamma=1-\sqrt{a(2-a)}\) is decreasing. It follows that \(Q,u,c_1,c_2\), and each summand in the definition of \(C\),
are nondecreasing. Continuity at zero proves the first assertion.

To prove \eqref{eq:stability-derivative}, fix \(y\in[5/2,3]\) and put
\[
 \alpha=\frac3{364},\quad \theta=\alpha y,\quad
 \theta_-=\frac{15}{728},\quad \theta_+=\frac9{364}.
\]
Since
\[
 \theta_-\leqslant\theta\leqslant\theta_+,\quad
 \frac17<s<\frac4{25},
\]
we have
\[
 \tau_-:=\frac{\theta_-}{1-1/7}<\tau<
            \frac{\theta_+}{1-4/25}=:\tau_+.
\]
By the definitions of \(m_1,m_2,w\), one has
\begin{align*}
 0.485&<\frac{1-\tau_+}{2}<m_1<
                 \frac{1-\tau_-}{2}<0.49,\\
 0.405&<\frac{1-4/25-\tau_+}{2}<m_2<
                 \frac{1-1/7-\tau_-}{2}<0.42,\\
 0.192&<\frac12\left(0.405-\frac{3\theta_+}{8(0.485)}\right)<w\\
 &<\frac12\left(\frac{1-1/7-\tau_-}{2}
                    -\frac{3\theta_-}{8(0.49)}\right)<0.201.
\end{align*}
Consequently,
\begin{align*}
 \eta&>\frac3{8(0.201)}>1.86,\\
 0.078&<\frac{1.86\,\theta_-}{0.49}<a<
                 \frac{3\theta_+}{8(0.192)(0.485)}<0.1,\\
 \gamma&<1-\sqrt{0.078(2-0.078)}<0.613,\\
 Q&>\frac3{4(0.49)}
                 \left(\frac12+\frac1{0.613}\right)>3.26.
\end{align*}

In the following, all auxiliary quantities are evaluated at $\theta=\alpha y$, and primes denote differentiation with respect to $y$. By direct computation,
\begin{align*}
 s'&=\frac{\alpha}{2s}>\frac{75}{2912},
 &\tau'&=\frac{\alpha(1-s/2)}{(1-s)^2}>\frac1{96},\\
 -m_1'&=\frac{\tau'}2>0.005,
 &-m_2'&=\frac{s'+\tau'}2>0.018.
\end{align*}
Using the preceding bounds, we obtain
\begin{align*}
 -w'&=\frac{-m_2'}2+\frac3{16}
       \left(\frac{\alpha}{m_1}
                 +\frac{\theta(-m_1')}{m_1^2}\right)>0.012,\\
 \eta'&=\frac{3(-w')}{8w^2}>0.11,\\
 a'&=\frac{\eta'\theta+\eta\alpha}{m_1}
                 +\frac{\eta\theta(-m_1')}{m_1^2}>0.036,\\
 -\gamma'&=\frac{(1-a)a'}{\sqrt{a(2-a)}}>0.074,\\
 Q'&=\frac{3(-m_1')}{4m_1^2}
                 \left(\frac12+\frac1\gamma\right)
                 +\frac{3(-\gamma')}{4m_1\gamma^2}>\frac13,\\
 u'&=\frac{Q'\theta+Q\alpha}{m_2}
                 +\frac{Q\theta(-m_2')}{m_2^2}>0.087.
\end{align*}

Set
\[
 r_1=\sqrt{1-a/2},\quad r_2=\sqrt{1-u},\quad D_i=1+r_i,
\]
and write \(c=T_1+T_2+T_3+1+T_4\), where
\[
 T_1=\frac{\eta}{D_1\sqrt{2m_1}},\quad
 T_2=\frac{2Q}{D_2\sqrt{2m_2}},\quad
 T_3=\frac1{1-s},\quad
 T_4=\frac{\theta}{(1-s)^2}.
\]
Since \(0<r_i\leqslant1\), \(D_i\leqslant2\), \(\sqrt{2m_1}<1\), and
\(\sqrt{2m_2}<0.917\), it follows that
\begin{align*}
 T_1'
 &=\frac{\eta'}{D_1\sqrt{2m_1}}
   +\frac{\eta a'}{4r_1D_1^2\sqrt{2m_1}}
   +\frac{\eta(-m_1')}{D_1(2m_1)^{3/2}}\\
 &>\frac{0.11}{2}+\frac{1.86\times 0.036}{16}
                       +\frac{1.86\times0.005}2>0.06,\\[2pt]
 T_2'
 &=\frac{2Q'}{D_2\sqrt{2m_2}}
   +\frac{Qu'}{r_2D_2^2\sqrt{2m_2}}
   +\frac{2Q(-m_2')}{D_2(2m_2)^{3/2}}\\
 &>\frac{1/3}{0.917}+\frac{3.26\times 0.087}{4 \times 0.917}
                       +\frac{3.26 \times 0.018}{0.84 \times 0.917}>0.51,\\[2pt]
 T_3'&=\frac{s'}{(1-s)^2}>
                   \frac{75/2912}{(6/7)^2}>0.035,\\
 T_4'&=\frac{\alpha}{(1-s)^2}
              +\frac{2\theta s'}{(1-s)^3}>
                   \frac{3/364}{(6/7)^2}>0.011.
\end{align*}
Thus,
\[
 c'(y)>0.06+0.51+0.035+0.011=0.616>\frac35.
\]
Finally, the previously established bound \(C(9/364)<7.326\) (see Remark \ref{remark})
implies
\[
 c(y)=c(3)-\int_y^3c'(v)\,dv
          <7.326-\frac35(3-y).
 \qedhere
\]
\end{proof}

\begin{lemma}\label{lem:scalar-comparison}
Let \(K=182\), \(n\geqslant3\), and \(X=(n-2)/n\). For \(y\in[0,3]\), set
\[
 Q_n(y)=\sqrt{\frac43}-\varepsilon_{182,n}(y),\quad
 F_n(y)=Q_n(y)^2-U_{182,X}(y).
\]
Then
\begin{equation}\label{eq:scalar-comparison}
 Q_n(y)>0,\quad F_n(y)>\frac1{141120}
 \quad\text{for every }y\in[0,3].
\end{equation}
\end{lemma}

\begin{proof}
Fix \(n\geqslant3\) and \(y\in[0,3]\). Write
\[
 B=B_{182,X}(y),\quad A=A_{182,X}(y),\quad
 \varepsilon=\varepsilon_{182,n}(y),\quad U=U_{182,X}(y).
\]

Since \(t=5/3+4/182\) and \(X\leqslant1\),
\begin{equation}\label{eq:uniform-U}
 U=\frac{4t(3-y)}{182}\left(\frac23+\frac X{182}\right)
       \leqslant\frac{3-y}{40}.
\end{equation}
By Lemma~\ref{lem:stability-slope} and the inequality \(C(9/364)<7.326\),
\begin{equation}\label{eq:uniform-B}
 B\leqslant\frac{7.326\,y}{182}+\frac3{8  \times 182^2}.
\end{equation}

We first consider the case \(n=3\). In this case, \(X=1/3\),
\(\kappa_3^2=4/3\), and \(3\lambda_3X/8=17/12\).
By \eqref{eq:uniform-B}, one has \(B,A<0.121\). Hence,
\[
 \begin{aligned}
 \varepsilon
 &\leqslant\sqrt{\frac43  \times 0.121 +\frac{17}{12}  \times 0.121^2}
             +3 \times 0.121+\frac{0.121^{3/2}}{\sqrt2}\\
 &<0.83.
 \end{aligned}
\]
This together with \eqref{eq:uniform-U} and \(\sqrt{4/3}>1.15\) yields
\[
 Q_3(y)>0.32,\quad
 F_3(y)>(0.32)^2-\frac3{40}
       =0.0274>\frac1{141120}.
\]

Now we consider the case \(n\geqslant4\). In this case,
\[
 \frac12\leqslant X<1,\quad
 \kappa_n^2=\frac43-\frac{2X}{3},\quad \lambda_n=8+10X.
\]
Define
\begin{equation}\label{eq:P-definition}
 P(b)=\sqrt{\frac23b+\frac{27}{4}b^2}
                 +3\sqrt3\,b+\frac3{\sqrt2}b^{3/2}.
\end{equation}
We claim that
\begin{equation}\label{eq:P-majorant}
 \varepsilon\leqslant P(b)
 \quad\text{whenever}\quad
 b\geqslant\max\left\{B,\frac1{182},\frac{16}{207}\right\}.
\end{equation}
In fact, we have \(A\leqslant b\), and the function
\[
 R_X(b)=\left(\frac43-\frac{2X}{3}\right)b
                    +\frac38(8+10X)Xb^2
\]
is convex in \(X\in[1/2,1]\). Moreover,
\[
 R_1(b)-R_{1/2}(b)
       =b\left(\frac{69}{16}b-\frac13\right)\geqslant0.
\]
It follows that
\[
 R_X(b)\leqslant R_1(b)=\frac23b+\frac{27}{4}b^2.
\]
Substitution into the definition of \(\varepsilon\)  proves \eqref{eq:P-majorant}.

If \(0\leqslant y\leqslant2\), by \eqref{eq:uniform-B} one has \(B,A<0.081\).
Using \(\kappa_n^2\leqslant1\) and \(3\lambda_nX/8\leqslant27/4\), we obtain
\[
 \begin{aligned}
 \varepsilon
 &\leqslant\sqrt{0.081+\frac{27}{4}  \times 0.081^2}
              +3\sqrt3 \times 0.081+\frac3{\sqrt2} \times 0.081^{3/2}\\
 &<0.83.
 \end{aligned}
\]
Thus, as in the case \(n=3\),
\[
 Q_n(y)>0.32,\quad F_n(y)>0.0274>\frac1{141120}.
\]

If \(2\leqslant y\leqslant5/2\), then $B,A<0.101$, and by
\eqref{eq:P-majorant},
\[
 \varepsilon\leqslant P(0.101)<0.97.
\]
Since \(U\leqslant1/40\), it follows that
\[
 Q_n(y)>0.18,\quad
 F_n(y)> 0.18^2-\frac1{40}
        =0.0074>\frac1{141120}.
\]

It remains to consider \(5/2\leqslant y\leqslant3\). Set
\[
 z=3-y,\quad b_*=\frac{12077}{100000},\quad
 b=b_*-\frac6{125}z.
\]
Then \(0\leqslant z\leqslant1/2\), and \eqref{eq:stability-affine} implies
\begin{align*}
 B
 &\leqslant\frac y{182}\left(7.326-\frac35z\right)
                             +\frac3{8  \times 182^2}\\
 &=\frac{3 \times 7.326}{182}+\frac3{8 \times 182^2}
        -\frac{7.326+(3/5)y}{182}z
 \leqslant b.
\end{align*}
Here we have used
\[
 \frac{3 \times 7.326}{182}+\frac3{8 \times 182^2}<b_*,
 \quad
 \frac{7.326+3/2}{182}>\frac6{125}.
\]
We also have
\[
 b\geqslant b_*-\frac3{125}=0.09677
      >\frac{19}{200}>\frac{16}{207}>\frac1{182}.
\]
Consequently, \eqref{eq:P-majorant} applies.

For \(v\geqslant19/200\), by differentiating \eqref{eq:P-definition} we get
\[
 \begin{aligned}
 P'(v)
 &=\frac{2/3+(27/2)v}{2\sqrt{(2/3)v+(27/4)v^2}}
                 +3\sqrt3+\frac9{2\sqrt2}\sqrt v\\
 &\geqslant\frac{9\sqrt3}{2}+\frac{9\sqrt{19}}{40}
 >\frac{35}{4}.
 \end{aligned}
\]
At \(b=b_*\), by a direct comparison of the squares of the nonnegative
terms, we have
\[
 \sqrt{\frac23b_*+\frac{27}{4}b_*^2}<0.42305,\quad
 3\sqrt3\,b_*<0.62754,\quad
 \frac3{\sqrt2}b_*^{3/2}<0.08904.
\]
In particular, \(P(b_*)<1.1397\). Hence,
\[
 \begin{aligned}
 \varepsilon
 &\leqslant P(b)
 =P(b_*)-\int_b^{b_*}P'(v)\,dv\\
 &\leqslant P(b_*)-\frac{35}{4}(b_*-b)
 <1.1397-\frac{21}{50}z.
 \end{aligned}
\]
Since \(\sqrt{4/3}>1.1547\), we conclude that
\begin{equation}\label{eq:positive-lower-bound}
 Q_n(y)>\frac3{200}+\frac{21}{50}z>0.
\end{equation}
Combining \eqref{eq:uniform-U} and
\eqref{eq:positive-lower-bound}, we obtain
\[
 \begin{aligned}
 F_n(y)
 &>\left(\frac3{200}+\frac{21}{50}z\right)^2-\frac z{40}\\
 &=\frac{441}{2500}\left(z-\frac{31}{882}\right)^2
                          +\frac1{141120}
 \geqslant\frac1{141120}.
 \end{aligned}
\]
This proves \eqref{eq:scalar-comparison} in all cases.
\end{proof}

Now we continue the proof of Theorem \ref{thm:main}.

Let \(p\) be the point chosen above, and set
\[
 W=\frac{|\Phi(h(p))|}{n\sqrt{n-2}}.
\]
The estimates in Step 2 and Lemma~\ref{lem:scalar-comparison} yield
\[
 W\geqslant Q_n(y)>0,\quad W^2\leqslant U_{182,X}(y)<Q_n(y)^2,
\]
which is impossible. Therefore  $g>{(n-2)}/{182}$, and hence
\[
 S_{\max}>\frac{2n}{3}+\frac{n-2}{182}.
 \qedhere
\]

\end{proof}

\appendix

\section{Normal directions to the Pauli orbit}
\label{app:pauli-normal}

In this appendix, we derive equations characterizing the normal space of the Pauli orbit at its standard representative.

Let
\[
 \mathcal V:=\bigl(\Sym_0^2(\RR^n)\bigr)^q,
 \quad
 c:=\sqrt{\frac n6},
\]
and set
\[
 P_1:=
 \begin{pmatrix}
 1&0\\
 0&-1
 \end{pmatrix},
 \quad
 P_2:=
 \begin{pmatrix}
 0&1\\
 1&0
 \end{pmatrix}.
\]
Consider the Pauli tuple
\[
 H=(H_1,\ldots,H_q)\in\mathcal V,
 \quad
 H_1=cP_1\oplus0,
 \quad
 H_2=cP_2\oplus0,
 \quad
 H_\mu=0\quad(\mu\geqslant3),
\]
and its orbit
\[
 E_{2n/3}:=\bigl(O(n)\times O(q)\bigr)\cdot H.
\]
For
\[
 d=(D_1,\ldots,D_q)\in\mathcal V,
\]
use the decomposition
\[
 \RR^n=\RR^2\oplus\RR^{n-2}
\]
to write
\[
 D_\alpha=
 \begin{pmatrix}
 B_\alpha&C_\alpha\\
 C_\alpha^T&E_\alpha
 \end{pmatrix}
\]
with
\begin{equation*}
 B_\alpha \in\Sym^2(\RR^2),\quad C_\alpha \in\RR^{2\times (n-2)},\quad E_\alpha\in\Sym^2(\RR^{n-2}).
 \end{equation*}

\begin{lemma}
$ d\perp T_HE_{2n/3}$ if and only if
\[
 \begin{cases}
 \langle B_1,P_2\rangle-\langle B_2,P_1\rangle=0,\\
 (D_1)_{1a}+(D_2)_{2a}=0,
 &a\geqslant3,\\
 (D_2)_{1a}-(D_1)_{2a}=0,
 &a\geqslant3,\\
 \langle B_\mu,P_1\rangle
 =\langle B_\mu,P_2\rangle=0,
 &\mu\geqslant3.
 \end{cases}
\]
\end{lemma}

\begin{proof}
The action of $O(n)\times O(q)$ on $\mathcal V$ is
\[
 (P,Q)\cdot(A_1,\ldots,A_q)
 =
 \bigg(
 \sum_\beta Q_{1\beta}PA_\beta P^T,\ldots,
 \sum_\beta Q_{q\beta}PA_\beta P^T
 \bigg).
\]
For
\[
 X\in\mathfrak{so}(n),
 \quad
 Y\in\mathfrak{so}(q),
 \quad
 P(s)=e^{sX},
 \quad
 Q(s)=e^{sY},
\]
Differentiating at $s=0$, we get
\[
 \left.
 \frac{\mathrm d}{\mathrm ds}
 \right|_{s=0}
 \bigl(P(s),Q(s)\bigr)\cdot H
 =
 \bigg(
 [X,H_\alpha]+\sum_\beta Y_{\alpha\beta}H_\beta
 \bigg)_{\alpha=1}^q.
\]
We can prove this by computing differentiation of each component of $(P(s),Q(s))\cdot H$.
Consequently,
\[
 T_HE_{2n/3}
 =
 \bigg\{
 \bigg(
 [X,H_\alpha]+\sum_\beta Y_{\alpha\beta}H_\beta
 \bigg)_{\alpha=1}^q
 :
 X\in\mathfrak{so}(n),\
 Y\in\mathfrak{so}(q)
 \bigg\}.
\]

With
\[
 \langle d,\widetilde d\rangle
 :=
 \sum_\alpha\operatorname{tr}
 \bigl(D_\alpha^T\widetilde D_\alpha\bigr),
\]
we have that
\[
 d\perp T_HE_{2n/3}
\]
if and only if
\begin{equation*}
 \sum_\alpha
 \bigg\langle
 D_\alpha,
 [X,H_\alpha]+\sum_\beta Y_{\alpha\beta}H_\beta
 \bigg\rangle
 =0
\end{equation*}
for every
$X\in\mathfrak{so}(n)$ and $Y\in\mathfrak{so}(q)$.

\medskip
\medskip

Since $Y_{\beta\alpha}=-Y_{\alpha\beta}$,
\[
 \begin{aligned}
 \sum_{\alpha,\beta}
 Y_{\alpha\beta}\langle D_\alpha,H_\beta\rangle
 =
 \sum_{\alpha<\beta}Y_{\alpha\beta}
 \bigl(
 \langle D_\alpha,H_\beta\rangle
 -\langle D_\beta,H_\alpha\rangle
 \bigr).
 \end{aligned}
\]
Thus this expression vanishes for every
$Y\in\mathfrak{so}(q)$ if and only if
\[
 \langle D_\alpha,H_\beta\rangle
 =
 \langle D_\beta,H_\alpha\rangle
 \quad(\alpha<\beta).
\]
Using
\[
 H_1=cP_1\oplus0,
 \quad
 H_2=cP_2\oplus0,
 \quad
 H_\mu=0\quad(\mu\geqslant3),
\]
these conditions become
\[
 \langle B_1,P_2\rangle-\langle B_2,P_1\rangle=0,
\]
\[
 \langle B_\mu,P_1\rangle=0,
 \quad
 \langle B_\mu,P_2\rangle=0
 \quad(\mu\geqslant3).
\]
All conditions with $\alpha,\beta\geqslant3$ are automatic.

\medskip

\medskip

Put $m=n-2$ and write an arbitrary
$X\in\mathfrak{so}(n)$ as
\[
 X=
 \begin{pmatrix}
 \omega J&U\\
 -U^T&Z
 \end{pmatrix},
 \quad
 J=
 \begin{pmatrix}
 0&1\\
 -1&0
 \end{pmatrix},
 \quad
 U=
 \begin{pmatrix}
 u^T\\
 v^T
 \end{pmatrix},
\]
where
\[
 \omega\in\RR,
 \quad
 u=(u_3,\ldots,u_n)^T,\quad
 v=(v_3,\ldots,v_n)^T\in\RR^m,
 \quad
 Z\in\mathfrak{so}(m).
\]
Since $H_1$ and $H_2$ vanish on $\RR^m$, the block $Z$ does not
contribute to $[X,H_\alpha]$.  Moreover,
\[
 [J,P_1]=-2P_2,
 \quad
 [J,P_2]=2P_1,
\]
and
\[
 [X,H_i]
 =
 c
 \begin{pmatrix}
 \omega[J,P_i]&-P_iU\\
 -U^TP_i&0
 \end{pmatrix}
 \quad(i=1,2).
\]
Hence,
\[
 \begin{aligned}
 \sum_{\alpha=1}^q
 \langle D_\alpha,[X,H_\alpha]\rangle
 =2c\bigg[
 &-\omega
 \bigl(
 \langle B_1,P_2\rangle-\langle B_2,P_1\rangle
 \bigr)\\
 &-\sum_{a=3}^n
 u_a\bigl((D_1)_{1a}+(D_2)_{2a}\bigr)\\
 &+\sum_{a=3}^n
 v_a\bigl((D_1)_{2a}-(D_2)_{1a}\bigr)
 \bigg].
 \end{aligned}
\]
Because $\omega$, $u$, and $v$ are arbitrary, this expression vanishes for
every $X\in\mathfrak{so}(n)$ if and only if
\[
 \langle B_1,P_2\rangle-\langle B_2,P_1\rangle=0,
\]
\[
 (D_1)_{1a}+(D_2)_{2a}=0,
 \quad
 (D_2)_{1a}-(D_1)_{2a}=0
 \quad(a\geqslant3).
\]

Since $X$ and $Y$ vary independently, the conditions obtained above are
jointly equivalent to the vanishing of the inner product with every element
of $T_HE_{2n/3}$.  Hence they are equivalent to
$d\perp T_HE_{2n/3}$, as required.
\end{proof}

\section{Normal equations at the Pauli orbit}
\label{app:normal-equations}

This appendix gives an explicit parametrization of the normal space at a
Pauli point and derives its decomposition into the radial direction and six
mutually orthogonal subspaces.

\medskip
\noindent\textbf{B.1 Notation and normal equations}

Let $n\geqslant 3$, $q\geqslant 2$, and $m=n-2$. Write $\Sym^2(\RR^k)$ for the
space of real symmetric $k\times k$ matrices and $\Sym_0^2(\RR^k)$ for
its trace-free subspace. Equip
\[
 \mathcal{V}=\bigl(\Sym_0^2(\RR^n)\bigr)^q
 \quad\text{with}\quad
 \ip{d}{\widetilde d}
 =\sum_{\alpha=1}^q\tr(D_\alpha\widetilde D_\alpha),
 \quad d=(D_\alpha),\quad\widetilde d=(\widetilde D_\alpha).
\]

Relative to $\RR^n=\RR^2\oplus\RR^m$, set
\[
 P_0=I_2,\quad
 P_1=\begin{pmatrix}1&0\\0&-1\end{pmatrix},\quad
 P_2=\begin{pmatrix}0&1\\1&0\end{pmatrix},
\]
and fix the Pauli point
\begin{equation}\label{eq:app-pauli}
 H_1=cP_1\oplus 0,\quad
 H_2=cP_2\oplus 0,\quad
 H_\mu=0\quad(3\leqslant\mu\leqslant q),\quad c=\sqrt{n/6}.
\end{equation}

Let 
\[
 \NH=(T_H E_{2n/3} )^\perp\subset\mathcal{V},
 \quad e_0=\frac{H}{\norm{H}}.
\]
We call the indices $1,2$ \emph{active} and the
indices $3,\ldots,q$ \emph{inactive}.

For $d=(D_\alpha)\in\mathcal{V}$, write
\begin{equation*}
 D_\alpha=
 \begin{pmatrix}B_\alpha&C_\alpha\\C_\alpha^T&E_\alpha\end{pmatrix},
\end{equation*}
where
\begin{equation*}
 B_\alpha \in\Sym^2(\RR^2),\quad C_\alpha \in\RR^{2\times m},\quad E_\alpha\in\Sym^2(\RR^m).
 \end{equation*}

In particular, $\tr B_\alpha+\tr E_\alpha=0$. By Appendix \ref{app:pauli-normal}, the normal equations are
\begin{equation}\label{eq:app-normal}
 \begin{cases}
 \ip{B_1}{P_2}=\ip{B_2}{P_1},\\
 (C_1)_{1j}+(C_2)_{2j}=0,&1\leqslant j\leqslant m,\\
 (C_2)_{1j}-(C_1)_{2j}=0,&1\leqslant j\leqslant m,\\
 \ip{B_\mu}{P_1}=\ip{B_\mu}{P_2}=0,&3\leqslant\mu\leqslant q.
 \end{cases}
\end{equation}

\medskip
\noindent\textbf{B.2 Explicit solution}

For $\alpha=1,2$, decompose $E_\alpha$ into its trace-free and scalar parts as
\[
E_\alpha
 =E_\alpha^0+\frac{\tr E_\alpha}{m}I_m.
\]

\begin{proposition}
\label{prop:app-parametrization}
A tuple $d=(D_\alpha)\in\mathcal{V}$ belongs to $\NH$ if and only if
its active components have the form
\begin{equation}\label{eq:app-active}
 \begin{aligned}
 D_1&=
 \begin{pmatrix}
 r_1P_0+x_1P_1+sP_2&C_1\\
 C_1^T&E_1^0-\dfrac{2r_1}{m}I_m
 \end{pmatrix},\\[2mm]
 D_2&=
 \begin{pmatrix}
 r_2P_0+sP_1+y_2P_2&C_2\\
 C_2^T&E_2^0-\dfrac{2r_2}{m}I_m
 \end{pmatrix},
 \end{aligned}
\end{equation}
where
\begin{equation}\label{eq:app-active-mixed}
 C_1=\begin{pmatrix}u^T\\v^T\end{pmatrix},\quad
 C_2=\begin{pmatrix}v^T\\-u^T\end{pmatrix},
 \quad u,v\in\RR^m,
\end{equation}
and its inactive components have the form
\begin{equation}\label{eq:app-inactive}
 D_\mu=
 \begin{pmatrix}
 -\dfrac12(\tr E_\mu)I_2&C_\mu\\
 C_\mu^T&E_\mu
 \end{pmatrix},\quad 3\leqslant\mu\leqslant q.
\end{equation}
The parameters
\[
 \begin{gathered}
 r_1,r_2,x_1,y_2,s\in\RR,\\
 E_1^0,E_2^0\in\Sym_0^2(\RR^m),\\
  u,v\in\RR^m,\\
 E_\mu\in\Sym^2(\RR^m),\quad
 C_\mu\in\RR^{2\times m}\quad(3\leqslant\mu\leqslant q)
 \end{gathered}
\]
are independent and are uniquely determined by $d$.
\end{proposition}

\begin{proof}
The matrices $P_0,P_1,P_2$ form an orthogonal basis of
$\Sym^2(\RR^2)$, with $\ip{P_i}{P_j}=2\delta_{ij}$ for
$0\leqslant i,j\leqslant 2$. Thus the first equation in \eqref{eq:app-normal}
identifies the $P_2$ coefficient of $B_1$ with the $P_1$ coefficient of
$B_2$. Denoting their common value by $s$, we obtain
\[
 B_1=r_1P_0+x_1P_1+sP_2,\quad
 B_2=r_2P_0+sP_1+y_2P_2.
\]

Since $\tr D_\alpha=0$, we have
\[
 E_\alpha
 =E_\alpha^0+\frac{\tr E_\alpha}{m}I_m
 =E_\alpha^0-\frac{2r_\alpha}{m}I_m.
\]

Writing the rows of $C_1$ as $u^T,v^T$, the second and third equations
in \eqref{eq:app-normal} force the rows of $C_2$ to be $v^T,-u^T$.
This proves \eqref{eq:app-active} and \eqref{eq:app-active-mixed}.

For an inactive index $\mu$, the last equation in
\eqref{eq:app-normal} implies $B_\mu=b_\mu I_2$. By the  trace-free
condition,
\[
 2b_\mu+\tr E_\mu=0,
 \quad b_\mu=-\frac12\tr E_\mu,
\]
which proves \eqref{eq:app-inactive}. No equation restricts
$E_\mu\in\Sym^2(\RR^m)$ or $C_\mu\in\RR^{2\times m}$.

Conversely, the displayed formulas define trace-free symmetric matrices
and satisfy every equation in \eqref{eq:app-normal}. Uniqueness follows
from the uniqueness of the coefficients in the Pauli basis, of the
trace decomposition, and of the block entries.
\end{proof}

\medskip
\noindent\textbf{B.3 The radial direction and the six normal blocks}

Introduce
\begin{equation}\label{eq:app-radial-coordinates}
 a=\frac{x_1+y_2}{2},\quad
 t=\frac{x_1-y_2}{2},
 \quad x_1=a+t,\quad y_2=a-t.
\end{equation}
The trace-free parts of the active upper blocks then satisfy
\begin{equation}\label{eq:app-radial-split}
 (B_1-r_1P_0,\,B_2-r_2P_0)
 =a(P_1,P_2)+t(P_1,-P_2)+s(P_2,P_1).
\end{equation}
After embedding these $2\times2$ matrices into the active components of
$\mathcal{V}$, the $a$ term, i.e., the embedding of the first term on the right hand side of \eqref{eq:app-radial-split}, is $(a/c)H$, and hence spans $\RR e_0$.

Using \eqref{eq:app-radial-coordinates} in
Proposition~\ref{prop:app-parametrization}, define $\V{1},\ldots,\V{6}$
by allowing precisely the parameters in the middle column below to vary
and setting all other parameters to zero:
\bigskip
\begin{center}
\begin{tabular}{@{}lll@{}}
\toprule
Subspace & Free parameters & Dimension\\
\midrule
$\V{1}$ & $E_1^0,E_2^0\in\Sym_0^2(\RR^m)$ & $n(n-3)$\\[1mm]
$\V{2}$ & $u,v\in\RR^m$ & $2(n-2)$\\[1mm]
$\V{3}$ & $r_1,r_2\in\RR$ & $2$\\[1mm]
$\V{4}$ & $t,s\in\RR$ & $2$\\[1mm]
$\V{5}$ & $(E_\mu)_{\mu=3}^q\in\bigl(\Sym^2(\RR^m)\bigr)^{q-2}$
 & $\dfrac{(q-2)(n-1)(n-2)}{2}$\\[2mm]
$\V{6}$ & $(C_\mu)_{\mu=3}^q\in\bigl(\RR^{2\times m}\bigr)^{q-2}$
 & $2(q-2)(n-2)$\\
\bottomrule
\end{tabular}
\end{center}
 
\bigskip

Thus $\V{1}$ consists of active trace-free lower blocks, $\V{2}$ of
active mixed blocks, and $\V{3}$ of active scalar upper blocks with
their trace-compensating lower blocks. The space $\V{4}$ is spanned by
the two nonradial Pauli directions in \eqref{eq:app-radial-split} (i.e., two embeddings of $(P_1,-P_2)$ and $(P_2, P_1)$).
Finally, $\V{5}$ and $\V{6}$ consist of inactive lower blocks (with the
required scalar upper blocks) and inactive mixed blocks, respectively.

\begin{corollary}
$\NH$ admits a unique orthogonal decomposition
\begin{equation}\label{eq:app-decomposition}
 \NH=\RR e_0\orthsum\V{1}\orthsum\V{2}\orthsum\V{3}
 \orthsum\V{4}\orthsum\V{5}\orthsum\V{6}.
\end{equation}
More precisely, the subspaces supported on active and inactive indices
satisfy
\begin{equation*}
 \begin{aligned}
 \NH^{\mathrm{active}}
 &=\RR e_0\orthsum\V{1}\orthsum\V{2}\orthsum\V{3}\orthsum\V{4},\\
 \NH^{\mathrm{inactive}}&=\V{5}\orthsum\V{6}.
 \end{aligned}
\end{equation*}
\end{corollary}

\begin{proof}
The parametrization and the invertible change of variables
\eqref{eq:app-radial-coordinates} prove that the indicated subspaces
span $\NH$ and that the resulting sum is direct. In particular, for
each inactive component,
\[
 D_\mu=
 \begin{pmatrix}
 -\dfrac12(\tr E_\mu)I_2&0\\0&E_\mu
 \end{pmatrix}
 +\begin{pmatrix}0&C_\mu\\C_\mu^T&0\end{pmatrix},
\]
where the two summands are the $\mu$ components of the $\V{5}$ and
$\V{6}$ parts, respectively.

To verify orthogonality, observe that for symmetric block matrices
\[
 \ip{D}{\widetilde D}
 =\tr(B\widetilde B)
  +2\tr(C^T\widetilde C)
  +\tr(E\widetilde E).
\]
Consequently, diagonal blocks are orthogonal to mixed blocks, and
trace-free lower blocks are orthogonal to scalar lower blocks. The
relations $\ip{P_i}{P_j}=2\delta_{ij}$ separate the scalar upper
blocks from the Pauli directions. Within the remaining active upper
blocks,
\[
 \begin{aligned}
 \ip{(P_1,P_2)}{(P_1,-P_2)}&=2-2=0,\\
 \ip{(P_1,P_2)}{(P_2,P_1)}&=0.
 \end{aligned}
\]
Finally, tuple components with distinct indices are orthogonal.
These observations prove all the asserted orthogonality relations.
\end{proof}

\begin{remark}
When $n=3$, one has $m=1$ and $\Sym_0^2(\RR)=\{0\}$, so
$\V{1}=\{0\}$. When $q=2$, the inactive families are empty and
$\V{5}=\V{6}=\{0\}$. The formulas above include both cases.
\end{remark}

\section{Direct computation of the normal quadratic form}
\label{app:direct-normal-spectrum}

Throughout this appendix, $H$ denotes the Pauli tuple in
\eqref{eq:app-pauli}, $c^2=n/6$, and $L:=D\Phi_H.$
We use the coordinates and the orthogonal decomposition of $\NH$ introduced
in \eqref{eq:app-active}--\eqref{eq:app-decomposition}.  The computation has
three steps.  We first list the nonzero components of $L$, then compute
$|Ld|^2$ on each of the seven normal summands, and finally verify that the
images of distinct summands are pairwise orthogonal.  These three steps give
the norm decomposition required in Lemma~\ref{lem:spectrum}.

For fixed $\alpha,k,l$, write
\[
 (Ld)_{\alpha;kl}:=
 \bigl((Ld)_{\alpha ijkl}\bigr)_{1\leqslant i,j\leqslant n}.
\]
Thus $(Ld)_{\alpha;kl}$ is the matrix formed by the first two tangent
indices.  Since $Ld$ is antisymmetric in $k,l$,
\begin{equation}\label{eq:ds-ordered-norm}
 |Ld|^2=2\sum_\alpha\sum_{k<l}|(Ld)_{\alpha;kl}|^2.
\end{equation}

\par\medskip
\noindent\textbf{C.1 The nonzero components of the linearization}

Write $H=(H_\gamma)\in\mathcal V$ and $d=(D_\gamma)\in\mathcal V$. Since
for $a=(A_\alpha)\in\mathcal V$ with $A_\alpha=(a^\alpha_{ij})$,
\[
G(a)_{ijkl}
=\sum_{\gamma=1}^{q}
\bigl(a^\gamma_{ik}a^\gamma_{jl}-a^\gamma_{il}a^\gamma_{jk}\bigr),
\quad
R^\perp(a)_{\alpha\beta kl}=[A_\alpha,A_\beta]_{kl},
\]
one has
\begin{equation*}
\begin{aligned}
\bigl((DG_H)[d]\bigr)_{ijkl}
&=\sum_{\gamma=1}^{q}
\Bigl((D_\gamma)_{ik}(H_\gamma)_{jl}
+(H_\gamma)_{ik}(D_\gamma)_{jl}\\
&\qquad
-(D_\gamma)_{il}(H_\gamma)_{jk}
-(H_\gamma)_{il}(D_\gamma)_{jk}\Bigr),\\[2mm]
\bigl((DR^\perp_H)[d]\bigr)_{\alpha\beta kl}
&=[D_\alpha,H_\beta]_{kl}+[H_\alpha,D_\beta]_{kl}.
\end{aligned}
\end{equation*}

Set
\[
 K(d):=Ld-\Phiv^0(d).
\]
Using the preceding component formulas and substituting
\eqref{eq:app-active}--\eqref{eq:app-inactive} into
\eqref{eq:normal-linearization}, we obtain the following complete
list of the nonzero components of $K(d)$.

\smallskip
\noindent\textit{The active pair $(k,l)=(1,2)$.}
\begin{equation}\label{eq:ds-components-12}
 \begin{aligned}
 K(d)_{1;12}
 &=c^2\begin{pmatrix}
 -2r_2I_2+2sP_1-(18a+2t)P_2&0\\0&-2E_2
 \end{pmatrix},\\
 K(d)_{2;12}
 &=c^2\begin{pmatrix}
 2r_1I_2+(18a-2t)P_1-2sP_2&0\\0&2E_1
 \end{pmatrix},\\
 K(d)_{\mu;12}
 &=2c^2\begin{pmatrix}
 0&P_1P_2C_\mu\\(P_1P_2C_\mu)^T&0
 \end{pmatrix}
 \quad(\mu\geqslant3).
 \end{aligned}
\end{equation}

\smallskip
\noindent\textit{The mixed pairs $k\leqslant2<l$.}
For $1\leqslant i,j,k\leqslant2$, $1\leqslant p,r\leqslant m$, and
$\alpha=1,2$,
\begin{equation}\label{eq:ds-components-mixed}
 \begin{aligned}
 K(d)_{\alpha,j,r+2,i,p+2}
 &=c^2\sum_{\beta=1}^2
 (P_\alpha P_\beta)_{ji}(E_\beta)_{rp},\\
 K(d)_{\mu,i,j,k,p+2}
 &=c^2\sum_{\beta=1}^2
 (P_\beta C_\mu)_{kp}(P_\beta)_{ij}.
 \end{aligned}
\end{equation}
The entries obtained by symmetry in the first two tangent indices and
antisymmetry in the last two are understood.  All other entries of $K(d)$
vanish.  In particular,
\begin{equation}\label{eq:ds-components-lower}
 (Ld)_{\alpha;kl}=\Phiv^0(d)_{\alpha;kl}
 \quad(3\leqslant k<l\leqslant n).
\end{equation}

We present the proof of the first line of \eqref{eq:ds-components-12} as an example.
Let $J=P_1P_2$.  The identities
\[
 P_1^2=P_2^2=I_2,\quad
 P_2P_1=-P_1P_2,\quad
 C_2=JC_1
\]
imply
\[
 [P_1,J]=2P_2,\quad
 [P_2,J]=-2P_1,\quad
 P_2C_1-P_1C_2=0.
\]
The first component in
\eqref{eq:ds-components-12} can be written as
\[
 K(d)_{1;12}
 =
 c^2
 \begin{pmatrix}
 -2[B_1,J]-2B_2-12aP_2
   &2JC_1-2C_2\\
 (2JC_1-2C_2)^T
   &-2E_2
 \end{pmatrix}.
\]
Since $C_2=JC_1$, the off-diagonal blocks vanish.  Moreover,
\[
 -2[B_1,J]-2B_2-12aP_2
 =
 -2r_2I_2+2sP_1-(18a+2t)P_2,
\]
which gives the first line of
\eqref{eq:ds-components-12}.

\par\medskip
\noindent\textbf{C.2 The squared norm on each normal summand}

We divide the sum in \eqref{eq:ds-ordered-norm} into three parts:
\[
 \begin{aligned}
 \mathcal E_{12}(d)
 &:=2\sum_\alpha |(Ld)_{\alpha;12}|^2,\\
 \mathcal E_{\mathrm{mix}}(d)
 &:=2\sum_\alpha\sum_{k\leqslant2<l}
       |(Ld)_{\alpha;kl}|^2,\\
 \mathcal E_{\mathrm{low}}(d)
 &:=2\sum_\alpha\sum_{3\leqslant k<l}
       |(Ld)_{\alpha;kl}|^2.
 \end{aligned}
\]
Thus,
\[
 |Ld|^2
 =\mathcal E_{12}(d)+\mathcal E_{\mathrm{mix}}(d)
  +\mathcal E_{\mathrm{low}}(d).
\]

For $d\in\V{2}$ or $d\in\V{5}$, by \eqref{eq:ds-components-12} and \eqref{eq:ds-components-mixed}, one has
$K(d)=0$. Therefore, Lemma~\ref{lem:flat-ricci-norm} implies
\begin{equation*}
 |Ld|^2=4n|d|^2
 \quad(d\in\V{2}\text{ or }d\in\V{5}).
\end{equation*}

For each of the remaining five summands, the following table records the
three quotients
\[
 \frac{\mathcal E_{12}(d)}{|d|^2},
 \qquad
 \frac{\mathcal E_{\mathrm{mix}}(d)}{|d|^2},
 \qquad
 \frac{\mathcal E_{\mathrm{low}}(d)}{|d|^2}.
\]
Each entry already includes the factor $2$ in
\eqref{eq:ds-ordered-norm}.
\begin{equation}\label{eq:ds-contraction-table}
 \begin{array}{c|ccc}
 &\mathcal E_{12}/|d|^2
 &\mathcal E_{\mathrm{mix}}/|d|^2
 &\mathcal E_{\mathrm{low}}/|d|^2\\ \hline
 \RR e_0&2(2-18c^2)^2&4m&0\\[1mm]
 \V{1}&8c^4&8\{(c^2-1)^2+c^4\}&4m\\[1mm]
 \V{3}&8c^4&\dfrac4n\{(n-2c^2)^2+4c^4\}&0\\[2mm]
 \V{4}&8(1-c^2)^2&4m&0\\[1mm]
 \V{6}&2(2c^2-1)^2&4(c^2-1)^2+2m+4&2(m-1)
 \end{array}
\end{equation}

\bigskip

A row is ignored when the corresponding summand is zero-dimensional.  The
normalizations used in the table are
\[
 \begin{array}{c|c}
 d\in\RR e_0&|d|^2=4a^2\\[1mm]
 d\in\V{1}&|d|^2=|E_1^0|^2+|E_2^0|^2\\[1mm]
 d\in\V{3}&|d|^2=\dfrac{2n}{m}(r_1^2+r_2^2)\\[2mm]
 d\in\V{4}&|d|^2=4(t^2+s^2)\\[1mm]
 d\in\V{6}&|d|^2=2\displaystyle\sum_{\mu\geqslant3}|C_\mu|^2.
 \end{array}
\]

We indicate the contractions that are not immediate from orthogonality of
the blocks.  

For the radial direction,
\[
 \begin{aligned}
 (Ld)_{1;12}&=(2-18c^2)a(P_2\oplus0),\\
 (Ld)_{2;12}&=-(2-18c^2)a(P_1\oplus0),\\
 \sum_\alpha\sum_{k\leqslant2<l}|(Ld)_{\alpha;kl}|^2
 &=8ma^2.
 \end{aligned}
\]

For $\V{1}$ and $\V{4}$, one uses
$\ip{P_\alpha}{P_\beta}=2\delta_{\alpha\beta}$ and
$\tr E_\alpha^0=0$. The lower contribution for $\V{1}$ is the
calculation in Lemma~\ref{lem:flat-ricci-norm}, restricted to the lower
$m\times m$ block.  For $\V{3}$, since $E_\alpha=-2r_\alpha I_m/m$, one has
\[
 \sum_\alpha\sum_{k\leqslant2<l}|(Ld)_{\alpha;kl}|^2
 =\frac4m\{(n-2c^2)^2+4c^4\}(r_1^2+r_2^2).
\]

Finally, for $\V{6}$, the Pauli identity
\[
 \sum_{\beta=1}^2(P_\beta C_\mu)_{kp}(P_\beta)_{ij}
 =\delta_{ik}(C_\mu)_{jp}+\delta_{jk}(C_\mu)_{ip}
  -\delta_{ij}(C_\mu)_{kp}
\]
gives the three sums
\[
 \begin{aligned}
 \sum_\mu|(Ld)_{\mu;12}|^2
 &=2(2c^2-1)^2\sum_\mu|C_\mu|^2,\\
 \sum_\mu\sum_{k\leqslant2<l}|(Ld)_{\mu;kl}|^2
 &=\{4(c^2-1)^2+2m+4\}\sum_\mu|C_\mu|^2,\\
 \sum_\mu\sum_{3\leqslant k<l}|(Ld)_{\mu;kl}|^2
 &=2(m-1)\sum_\mu|C_\mu|^2.
 \end{aligned}
\]
These are the sums before the outer factor $2$ in
\eqref{eq:ds-ordered-norm}.  Multiplying by that factor and dividing by
$|d|^2=2\sum_\mu|C_\mu|^2$, we obtain the last row of
\eqref{eq:ds-contraction-table}.

Adding the three columns in \eqref{eq:ds-contraction-table}, setting
$m=n-2$ and $c^2=n/6$, and including $\V{2}$ and $\V{5}$, we obtain
\bigskip
\[
 \begin{array}{c|c}
 \text{normal subspace}&|Ld|^2/|d|^2\quad(d\ne0)\\ \hline
 \RR e_0&n^2 \cdot \left(18-\dfrac{20}{n}\right)\\[2mm]
 \V{1}&n^2 \cdot \dfrac{2(n+2)}{3n}\\[2mm]
 \V{2}&n^2 \cdot \dfrac4n\\[2mm]
 \V{3}&n^2 \cdot \dfrac{2(n+10)}{9n}\\[2mm]
 \V{4}&n^2 \cdot \dfrac{2(n+6)}{9n}\\[2mm]
 \V{5}&n^2 \cdot \dfrac4n\\[2mm]
 \V{6}&n^2 \cdot \dfrac{n+4}{3n}.
 \end{array}
\]

\bigskip

In the notation of Lemma~\ref{lem:spectrum}, this is
\begin{equation}\label{eq:ds-eigenvalues}
 |L(xe_0)|^2=n^2\lambda_nx^2,
 \qquad
 |Ld|^2=n^2\mu_j|d|^2
 \quad(d\in\V{j}).
\end{equation}

\par\medskip
\noindent\textbf{C.3 Orthogonality of the normal images}

It remains to check that the squared norms in
\eqref{eq:ds-eigenvalues} have no cross terms.  Define
\[
 B_L(u,v):=\ip{Lu}{Lv}.
\]
Since $L=\Phiv^0+K$, we have
\[
 B_L(u,v)={}\ip{\Phiv^0(u)}{\Phiv^0(v)} +\ip{\Phiv^0(u)}{K(v)}+\ip{K(u)}{\Phiv^0(v)}+\ip{K(u)}{K(v)}.
\]
By polarizing the identity in Lemma~\ref{lem:flat-ricci-norm}, we get
\[
 \ip{\Phiv^0(u)}{\Phiv^0(v)}=4n\ip{u}{v}.
\]
Thus the first term vanishes when $u$ and $v$ belong to distinct
summands of the orthogonal decomposition \eqref{eq:app-decomposition}.
The blockwise reasons for the remaining three terms are summarized
below.
\medskip
\begingroup\small\setlength{\arraycolsep}{4pt}
\[
 \begin{array}{c|l}
 \text{pair of summands}&\text{reason for vanishing}\\ \hline
 \V{2}\ ;\ \RR e_0,\V{1},\V{3},\V{4}
   &\text{diagonal blocks are orthogonal to mixed blocks}\\
 \V{1}\ ;\ \RR e_0,\V{3},\V{4}
   &\tr E_1^0=\tr E_2^0=0\\
 \V{3}\ ;\ \RR e_0,\V{4}
   &\tr P_1=\tr P_2=0\\
 \RR e_0\ ;\ \V{4}
   &\ip{(P_1,P_2)}{(P_1,-P_2)}=0,
     \quad\ip{P_1}{P_2}=0\\
 \text{active summands;inactive summands}
   &\text{their images have different normal indices}\\
 \V{5}\ ;\ \V{6}
   &\text{diagonal blocks are orthogonal to mixed blocks}.
 \end{array}
\]
\endgroup

\medskip
Here ``diagonal'' and ``mixed'' refer to the $2\times2$, $m\times m$
diagonal blocks and the $2\times m$ off-diagonal blocks, respectively.
By the explicit formulas in Lemma~\ref{lem:flat-ricci-norm} and
\eqref{eq:ds-components-12}--\eqref{eq:ds-components-lower},
$\Phi^0$ and $K$ preserve the block type of each normal summand.
Hence the blockwise orthogonality recorded in the table applies to the
three remaining terms in the expansion of $B_L(u,v)$. The listed cases
cover every pair of distinct summands. Therefore, if
$d=xe_0+\sum_{j=1}^6z_j$ with $z_j\in\V{j}$, then
\begin{equation}\label{eq:ds-full-quadratic-form}
 |Ld|^2=|L(xe_0)|^2+\sum_{j=1}^6|Lz_j|^2.
\end{equation}

Combining \eqref{eq:ds-eigenvalues} with
\eqref{eq:ds-full-quadratic-form}, we obtain
\[
 |Ld|^2
 =n^2\lambda_nx^2+n^2\sum_{j=1}^6\mu_j|z_j|^2.
\]
This is the norm decomposition used in the proof of
Lemma~\ref{lem:spectrum}.

\medskip

\section*{Acknowledgments}

This research was supported by the National Natural Science Foundation of China, Grant Nos. 12471051, 12171423, and 12071424.

\end{document}